\documentclass[12pt]{article}

\usepackage{graphicx}
\usepackage{latexsym,amssymb}
\usepackage{amsthm}
\usepackage{indentfirst}
\usepackage{amsmath}
\usepackage{color}
\usepackage{xcolor}
\usepackage{fourier}
\usepackage[colorlinks=true,backref=page]{hyperref}
\usepackage[all]{xy}

\usepackage[
a4paper,
text={160true mm,230true mm},
centering,
top=35true mm,
head=5true mm,
headsep=2.5true mm,
foot=8.5true mm]
{geometry}

\newtheorem{theoremalph}{Theorem}

\newtheorem*{Main Theorem}{Main Theorem}

\newtheorem{Theorem}{Theorem}[section]
\newtheorem*{Theorem A}{Theorem A}
\newtheorem*{Theorem A'}{Theorem A'}
\newtheorem*{Theorem B'}{Theorem B'}

\newtheorem*{Conjecture}{Conjecture}
\newtheorem{Definition}[Theorem]{Definition}
\newtheorem{Proposition}[Theorem]{Proposition}
\newtheorem{Lemma}[Theorem]{Lemma}

\newtheorem{Remark}[Theorem]{Remark}

\newtheorem{Remark-numbered}[Theorem]{Remark}
\newtheorem{Corollary}[Theorem]{Corollary}

\newtheorem*{Claim}{Claim}
\newtheorem{Claim-numbered}[Theorem]{Claim}

 \def\NN{{\mathbb N}}

    \def\cU{{\cal U}}
   \def\cP{{\cal P}}

\newcommand{\loc}{\operatorname{loc}}

\def\dim{\operatorname{dim}}

\def\diam{\operatorname{Diam}}

\def\supp{\operatorname{Supp}}

\newcommand\Prob{\mathbb P}

\newcommand\homrel[1]{\stackrel{\tiny #1}\sim}
\newcommand\htop{h_{\rm top}}

\newcommand\ignore[1]{}

\begin{document}

\title{Continuity properties of ergodic measures of maximal entropy for $C^r$ surface diffeomorphisms}

\author{J\'{e}r\^{o}me Buzzi, Chiyi Luo, and Dawei Yang\footnote
{D. Yang  was partially supported by National Key R\&D Program of China (2022YFA1005801), NSFC 12171348 \& NSFC 12325106, ZXL2024386 and Jiangsu Specially Appointed Professorship. C. Luo was partially supported by NSFC 12501244.}
}

\date{}
\maketitle

\begin{abstract} 
Let $f$ be a $C^r$ surface diffeomorphism with large entropy (more precisely, $h_{\rm top}(f)>\lambda_{\min}(f)/{r}$). 
Then the number of ergodic measures of maximal entropy is upper semicontinuous at $f$.
This generalizes the $C^\infty$ case studied in \cite{BCS22}, answering Question 1.9 there. 

Moreover, the number of such measures is locally constant if and only if every ergodic measure of maximal entropy of $f$ admits an ergodic continuation under small perturbations. 
In this case, the accumulation points of ergodic measures of maximal entropy are themselves ergodic. 
These facts are new, even in the $C^\infty$ case.
\end{abstract}

\tableofcontents	

\section{Introduction}\label{SEC:1}
Entropy plays a central role in understanding chaotic systems. 
Assume that $X$ is a compact metric space and $T:X\to X$ is a continuous map. 
The topological entropy of $T$, denoted by $h_{\rm top}(T)$, quantifies the complexity of a dynamical system by counting how the number of distinguishable orbits grow over time.
A classical theorem of Bogoliubov-Krylov \cite{BoK37} ensures the existence of at least one invariant probability measure.
For such measures, Kolmogorov and Sinai introduced the measure-theoretic entropy, or metric entropy, which quantifies the complexity of a dynamical system from the perspective of an invariant measure.
For a $T$-invariant probability measure $\mu$, we denote by $h_{\mu}(T)$ the metric entropy of $\mu$ with respect to $T$.
The topological entropy and the metric entropy are related by the following variational principle (see \cite{Dina71,Goo71}):
$$h_{\rm top}(T)=\sup\{h_\mu(T):\mu~\text{is}~T\text{-invariant}\}=\sup\{h_\mu(T):\mu~\text{is}~T\text{-ergodic}\}.$$
For systems with positive topological entropy, a special class of invariant measures is given by those for which the metric entropy equals the topological entropy of the system. Such measures are called \emph{measures of maximal entropy} (MME for short). 
Specifically, a $T$-invariant measure $\mu$ is said to be a measure of maximal entropy of $T$ if $h_\mu(T)=h_{\rm top}(T)$.

From now on we consider a $C^r$ surface diffeomorphism $f$, i.e., a $C^r$ smooth diffeomorphism of some compact, boundaryless, two-dimensional manifold $M$ (in this paper $r$ is a real number larger than $1$ (not only an integer), or $r=\infty$, see Remark~\ref{Rem:NonintegerR}).

If $f$ belongs to the large entropy surface setting (see Definition~\ref{Def:LargeEntropy}), Buzzi-Crovisier-Sarig  \cite{BCS22} have shown the finiteness of the number $N_f$ of ergodic MMEs.
This includes all $C^\infty$ diffeomorphisms $f$ of compact, boundaryless surfaces with $\htop(f)>0$, proving a well-known conjecture of Newhouse \cite{New91}.

The existence problem had been solved previously by a celebrated theorem of Newhouse \cite{New89}: all $C^\infty$ diffeomorphisms of compact manifolds have some MMEs.
More recently,  Burguet \cite{Bur24P} has established the existence of MMEs in the same large entropy surface setting, solving Conjecture 1 of \cite{BCS22}.

\medbreak

This article combines the techniques of \cite{BCS22} and \cite{Bur24P} to study how the ergodic MMEs depend on the diffeomorphism in the finite smoothness,  large entropy surface setting. 
Our proofs are carefully crafted to achieve the lower bound in \eqref{eqLargeEntropy}, that is, $\min(\lambda(f),\lambda(f^{-1}))/r$ rather than the larger $\max(\lambda(f),\lambda(f^{-1}))/r$. 

Our results generalize and make more precise the results obtained by \cite{BCS22} for $C^\infty$ diffeomorphisms of surfaces. 
In particular, Theorem~\ref{Thm:uniform-finite} solves the problem stated in \cite[Remark 1.9]{BCS22}. We give a family of examples showing that all drops allowed by Theorem~\ref{Thm:uniform-finite} actually occur. 
We also formulate some open problems in Section~\ref{Sec:Open}.

\subsection{Main results}
Given a diffeomorphism $f:M\rightarrow M$ on a compact  Riemannian manifold $M$, denote by $N_f$ the number of ergodic measures of maximal entropy of $f$, which may be zero, finite, or infinite.
For a compact $f$-invariant set $\Lambda$, one defines
\begin{equation}\label{eq:Lam}
	\lambda^u(\Lambda,f):=\lim_{n\to\infty}\frac{1}{n}\log\sup_{x\in\Lambda}\|D_xf^n\|,~~~\lambda^s(\Lambda,f):=\lim_{n\to\infty}\frac{1}{n}\log\sup_{x\in\Lambda}\|D_xf^{-n}\|.
\end{equation}
By taking $\Lambda=M$, one has important dynamical quantities\footnote{In \cite{LY24}, $\lambda^u(f)$ is denote by $R(f)$; hence $\lambda^s(f)$ is $R(f^{-1})$.}:
\begin{align}
	\lambda^u(f):=\lambda^u(M,f),&~~~\lambda^s(f):=\lambda^s(M,f); \label{eq:lambdaus} \\
	\lambda_{\min}(f):=\min\{\lambda^u(f),\lambda^s(f)\},&~~~\lambda_{\max}(f):=\max\{\lambda^u(f),\lambda^s(f)\}.\label{eq:lambdamin}
\end{align}

\begin{Definition}\label{Def:LargeEntropy}
A $C^r$ diffeomorphism $f$ belongs to the \emph{large entropy surface  setting}, if it acts on a compact, boundaryless two-dimensional manifold and satisfies 
 \begin{equation}\label{eqLargeEntropy}
     h_ {\rm top}(f) > \frac{\lambda_{\min}(f)}{r} .
  \end{equation}
Moreover we assume $r>1$ to be a real number or $r=\infty$ (in which case we set $\lambda_{\min}(f)/\infty:=0$).
\end{Definition}

\begin{Remark}\label{Rem:NonintegerR}
Note that we use the results of \cite[Lemma 12]{Bur24P}, which require $r>1$ to be an integer. 
However, without assuming $r$ is an integer, the reparametrization lemma still holds by following Burguet’s proof (see Luo-Yang \cite[Appendix]{LuY25}).
\end{Remark}

We generalize  to the large entropy surface setting the upper semicontinuity of the number of ergodic MMEs established by \cite[Theorem 4]{BCS22} in the case $r=\infty$. 
Note that this solves the problem in \cite[Remark 1.9]{BCS22}.

\begin{theoremalph}\label{Thm:uniform-finite}
Let $f$ be a $C^r$ diffeomorphism in the large entropy surface setting.
Then there is a  $C^r$  neighborhood $\mathcal U$ of $f$ such that for any $g\in\mathcal U$, one has that
$$1\le N_{g}\le N_{f}<\infty.$$
\end{theoremalph}

In particular, having a unique MME is an open property in the large entropy surface setting.

\begin{Remark}
In the case $r<\infty$, the techniques of \cite{BCS22} only yielded local boundedness, see Remark 1.9 there. 
A reparametrization lemma of Burguet \cite{Bur24P} will be key to our arguments.
\end{Remark}

It is obvious that the upper semicontinuous function $f\mapsto N_f$ cannot be continuous since it would then be locally constant. 
Let us be more precise:

\begin{Proposition}\label{Prop:Construction}
Given any two integers $1\le m\le n$ and any surface $S$, there is a smooth family $(f_t)_{|t|<1}$ of $C^\infty$-diffeomorphisms of that surface $S$, such that $\htop(f_{t})>0$ for all $t$, $N_{f_0}=n$, and $N_{f_t}=m$ for all $t>0$. 
\end{Proposition}

\medbreak

In the large entropy surface setting, we show that the \emph{ergodic} MMEs are continuous whenever their number is locally constant.

\begin{theoremalph}\label{Thm:equal}
Let $f$ be a $C^r$ diffeomorphism in the large entropy surface setting.
For any sequence of diffeomorphisms $f_n$ satisfying $f_n \rightarrow f$ in the $C^r$ topology, the following are equivalent:
 \begin{enumerate}
  \item[(1)] $N_{f_n}=N_f$ for all sufficiently large $n$; 
  \item[(2)] for any ergodic measure $\mu$ of maximal entropy of $f$, there are ergodic measures of maximal entropy $\mu_n$ of $f_n$ such that $\mu=\lim_n\mu_n$.
 \end{enumerate}
\end{theoremalph}

We show that the stability of the number of MMEs implies the stability of ergodicity for measures with entropy converging to the topological entropy.

\begin{theoremalph}\label{Thm:ergodicLimit}
Let $f$ be a $C^r$ diffeomorphism in the large entropy surface setting.
Fix some sequence of diffeomorphisms $f_n$ such that $f_n \rightarrow f$ in the $C^r$ topology. 
Assume that
$$ \forall n\ge1,\quad N_{f_n}=N_f. $$
Then we have (*) given any sequence of ergodic measures $\mu_n$ of $f_n$ with $\lim_n h_{\mu_n}(f_n)=\htop(f)$, any limit point of $(\mu_n)$ is an ergodic mesure of maximal entropy of $f$.
\end{theoremalph}

There is no converse to the above implication: property (*) does not imply that $N_{f_n}=N_f$ for all large $n$. 
See Proposition~\ref{Prop:Example-detailed}.

\begin{Remark}
Even the $C^\infty$ case of Theorem~\ref{Thm:equal} had not been considered before.
The special case of a constant sequence $f_n=f$ in Theorem~\ref{Thm:ergodicLimit} can be deduced from the study of SPR diffeomorphism in \cite{BCS25}. 
Indeed, any surface diffeomorphism with large entropy is SPR by \cite[Theorem 3.19]{BCS25} and therefore the conclusion of Theorem~\ref{Thm:ergodicLimit} holds by \cite[Theorem B]{BCS25}. 
We note that the proof of \cite[Theorem B]{BCS25} is completely different, e.g., it relies heavily on symbolic dynamics.
\end{Remark}

Theorems A, B, and C above will be deduced from the next result which generalizes Proposition~5.4 of \cite{BCS22} to the large entropy, $C^r$ setting with finite $r>1$.
We refer to Section~\ref{Sec:homoclinic-relation} for the notion of homoclinic relation. 

\begin{theoremalph}\label{Thm:HomoRel}
Let $f$ be a $C^r$ diffeomorphism in the large entropy surface setting.
Fix some sequence of diffeomorphisms $f_n$ such that $f_n \rightarrow f$ in the $C^r$ topology. 
Fix some $f_n$-ergodic measures $\mu_n$ converging to some $\mu_{0}$ in the weak $\ast$-topology.

Assume that $\lim_n h_{\mu_n}(f_n)=\htop(f)$. 
Then $\mu_{0}$ is an MME for $f$ and its ergodic decomposition is supported by finitely many hyperbolic measures. 
Let $\nu$ be one of them and let $O$ be some hyperbolic periodic orbit which is homoclinically related to~$\nu$. 

Then, for all large $n$, $\mu_n$ is homoclinically related to the hyperbolic continuation $O_n$ of $O$ with respect to $f_n$.
\end{theoremalph}

\subsection{Some open problems}\label{Sec:Open}
\paragraph{A continuity property for homoclinic classes}
Using \cite{Bur24P} we are able to control the entropy at small scale only near MMEs. 
But it is natural to expect some control as long as the entropy is larger than $\lambda_{\min}(f)/r$, provided that one considers, not invariant measures, but \emph{homoclinic classes of measures}, i.e., equivalence classes for the homoclinic relation among ergodic, hyperbolic measures. Let us recall some definitions (see Section~\ref{Sec:homoclinic-relation} for more background and references).

Let $C\subset\Prob_{\rm erg}(f)$ be some homoclinic class of measures for the diffeomorphism $f$. 
The support of $C\subset\Prob_{\rm erg}(f)$ is:  $\supp(C) := \overline{\bigcup_{\mu\in C} \supp(\mu)}$.
The entropy of $C\subset\Prob_{\rm erg}(f)$ is: $ h(f,C) := \sup \{h(f,\mu):\mu\in C\}$. 
We can now state a generalization of Theorem~\ref{Thm:uniform-finite}.

\medbreak

\begin{Conjecture}
Let $f$ be a $C^r$ diffeomorphism of a compact surface. Given $0<t\le \htop(f)$, let $\mathcal H(f,t)$ be the set of homoclinic classes of measures $C$ satisfying $h(f,C)\ge t$. Let $f_n\to f$ converge in the $C^r$ topology. 
Then, for any $t>\lambda_{\min}(f)/r$, for all large $n$,
 \begin{equation}\label{eq:dropHC}
     \#\mathcal H(f_n,t) \le  \#\mathcal H(f,t)\;.
  \end{equation}
\end{Conjecture}

Furthermore, one should be able to generalize Theorem~\ref{Thm:ergodicLimit} as follows. 
For a diffeomorphism $g$ that is $C^1$-close to $f$, we define the \emph{hyperbolic continuation} of some homoclinic class of measures $C$ of $f$ to be the homoclinic class $C_g$ of $g$ containing all the ergodic hyperbolic measures of $g$ that are homoclinically related to the hyperbolic continuation of some periodic orbit $O$ with $\mu_O\in C$, where  $\mu_O$ is an ergodic measure supported on a hyperbolic periodic orbit $O$.

\begin{Conjecture}
In the above setting, fix some $t>\lambda_{\min}(f)/r$. 
If $\#\mathcal H(f_n,t)=\#\mathcal H(f,t)$ for all large $n$, then, each homoclinic class in $\mathcal H(f_n,t)$ is the hyperbolic continuation of some homoclinic class of $\mathcal H(f,t)$, for all large $n$.
\end{Conjecture}

More precisely, one can relate any drop in \eqref{eq:dropHC} to homoclinic classes droping in entropy below $t$ or to the merging of several homoclinic classes.
More generally, we expect semicontinuity of the part of the spectral decomposition built in \cite{BCS22} above entropy $\lambda_{\min}(f)/r$.

\paragraph{A question about stability}
We have considered stability along a sequence $f_n$. 
There are stronger form of stability and we propose the following.

\paragraph{Question.} 
{\it Let $f$ be a $C^r$ diffeomorphism in the large entropy surface setting. 
	If $N_g=N_f$ for \emph{all} $g$ in a $C^r$ neighborhood of $f$, does it follow that the supports of all ergodic measures of maximal entropy of $f$ are pairwise disjoint?}

\paragraph{Other natural measures}

We are also interested in perturbing other natural measures.
For instance, it has been shown in \cite{BCS22} that for $C^r$ surface diffeomorphism $f$, for any $b>\lambda^u(f)/r$, there are at most finitely many SRB measures whose positive Lyapunov exponents are larger than $b$. 
Thus, we formulate the following conjecture.

\begin{Conjecture}
Given a $C^\infty$ surface diffeomorphism $f$ and $b>0$, there is a $C^\infty$ neighborhood $\mathcal U$ of $f$ such that for any $g\in\mathcal U$,
$$0\le N^{\textrm SRB}_{g}(b)\le N^{\textrm SRB}_{f}(b)<\infty,$$
where $N^{\textrm SRB}_{g}(b)$ is the number of SRB measures of $g$ with metric entropy larger than $b$.
\end{Conjecture}

It is worth noting that it is not even known whether $\sup_{g\in\mathcal U} N_g<\infty$.

\subsection{Outline the proof}\label{Sec:Ideas}
The key Theorem~\ref{Thm:HomoRel} studies how homoclinic relations change under perturbations. 
The general strategy of the proof is roughly similar to \cite{BCS22}: 
\begin{enumerate}
        \item one finds a scale at which the entropy can be computed with some arbitrarily small precision;
        \item one covers the support of limit measure $\mu_0$ by $su$-quadrilaterals, i.e., toplogical disks bounded by segments of stable and unstable laminations of $\mu_0$. 
                   This will show that large entropy measures $\mu$ close to $\mu_0$ must give rise to intersections between stable and unstable manifolds;
        \item the large transverse dimensions of invariant foliations associated to measures with large entropy imply that some of these intersections must be  transverse because of a variant of Sard's lemma.
\end{enumerate} 

However, the above steps have to be significantly modified with respect to the proof of \cite[Theorem 4]{BCS22}.

A fundamental difference occurs in Step 1. 
Indeed, there need not be a scale to uniformly approximate the entropy for all measures in contrast to the $C^\infty$ case where the tail entropy vanished.  
Thus, we use Burguet's reparametrizations \cite{Bur24P} in place of \cite{Buzzi-SIM} to get a uniform scale for unstable manifolds of measures with large entropy.

As a consequence,  Step 2 must consider separately the stable and the unstable laminations instead of their intersections (the small topological homoclinic classes in \cite{BCS22}). 
Thus, we show that both stable and unstable manifolds of $\mu$ must topologically cross the boundary of the $su$-quadrilaterals, hence the unstable and stable manifolds of $\mu_0$.

In Step 3, it may be that only the unstable laminations of measures with large entropy have large transverse dimensions (this property was called $u$-thickness in \cite{BCS22}). 
The dynamical Sard Lemma from \cite{BCS22} immediately gives transverse intersections between the unstable lamination of $\mu$ and  the stable manifold making up the stable boundary of the above $su$-quadrilaterals. 
This proves half the homoclinic relation: $\mu\preceq\mu_0$. 

But one cannot get the other half ($\mu_0\preceq\mu$) in a symmetric way since the stable lamination of $\mu$ may have a small transverse dimension. 
We need to use that the unstable boundary of the $su$-quadrilaterals belongs to the unstable \emph{lamination} of $\mu_0$ which does have a large transverse dimension.\footnote{Such an asymmetric argument was not necessary in the proof of Theorem 4 or Proposition 5.4 of \cite{BCS22}, but a related idea was used at the end of the proof of \cite[Theorem 5.2]{BCS22}.}

\subsection{Outline of the paper}\label{Sec:Outline}
We begin by some review of classical results in Section~\ref{Sec:NUH} and of more recent ones by Burguet and by Buzzi-Crovisier-Sarig in Section~\ref{Sec:revisited}.
The core of the paper is in Sections \ref{Sec:MRT}~and~\ref{Sec:MRT-geo}, where we prove the key Theorem~\ref{Thm:main-reduction}, a version of Theorem~\ref{Thm:HomoRel}.
In Section~\ref{Sec:ProofMainThms}, we prove a Corollary~\ref{Cor:LimitDisjoint} of Theorem~\ref{Thm:HomoRel} from which all the other main results follow easily (Theorems \ref{Thm:uniform-finite}, \ref{Thm:equal}, and \ref{Thm:ergodicLimit}).
In the last Section~\ref{Sec:Construction}, we build the examples announced as Proposition~\ref{Prop:Construction}.

\section{Non-uniformly hyperbolic dynamics}\label{Sec:NUH}

We review some classical results from the theory of nonuniform hyperbolic measures and their entropy.

\subsection{Lyapunov exponent}\label{Sec:LYE}

Let $f$ be a $C^1$ diffeomorphism on a compact manifold $M$. 
Recall the Oseledec theorem \cite{Ose68}: for an ergodic measure $\mu$ of $f$, there are finite numbers 
$$\lambda_1(\mu,f)>\lambda_2(\mu,f)>\cdots>\lambda_t(\mu,f)$$
and a $Df$-invariant splitting on a full $\mu$-measure set 
$$E^1\oplus E^2\oplus\cdots\oplus E^t$$
such that $\sum_{i=1}^t \dim E^i=\dim M$, and for any $1\leq i \leq t$ one has
$$\lim_{n\to\pm\infty}\frac{1}{n}\log\|D_{x}f^n|_{E^i(x)}\|=\lambda_i(\mu,f),~\mu\text{-a.e.}$$

Given an invariant measure of $\mu$ of $f$, one denotes
$$\lambda^+(\mu,f):=\lim_{n\to\infty}\frac{1}{n}  \int \log  \|D_xf^n\|{\rm d}\mu(x),~~~\lambda^-(\mu,f):=\lambda^+(\mu,f^{-1}).$$
When one considers a surface diffeomorphism $f$ and an ergodic measure $\mu$, then there is a $Df$-invariant measurable splitting $E^s\oplus E^u$, such that for $\mu$-almost every $x$,
$$\lim_{n\to\pm\infty}\frac{1}{n}\log\|Df^n|_{E^s(x)}\|=\lambda^-(\mu,f),~~~\lim_{n\to\pm\infty}\frac{1}{n}\log\|Df^n|_{E^u(x)}\|=\lambda^+(\mu,f)$$
If $\mu$ is only invariant, but $f$ is also a surface diffeomorphism, one has that
\begin{equation}\label{e.integrate-Lyapunov}
\lambda^-(\mu,f)+\lambda^+(\mu,f)=\int\log|{\rm Det}(D_xf)|{\rm d}\mu(x).
\end{equation}

An ergodic measure $\mu$ is said to be \emph{hyperbolic}, if for $\mu$-almost every point $x$, all Lyapunov exponents of $x$ are non-zero.  Thus, every point $x$ has a splitting $E^s(x)\oplus E^u(x)$ such that all Lyapunov exponents along  $E^s(x)$ are all negative, and  all Lyapunov exponents along  $E^u(x)$ are all positive.
By the Ruelle inequality \cite{Rue78}, one knows that if an ergodic measure $\mu$ of a surface diffeomorphism with positive metric entropy, then $\mu$ is a hyperbolic measure.

\subsection{Stable/unstable manifolds and homoclinic relations}\label{Sec:homoclinic-relation}

From Pesin theory \cite{BP07,Pes76,Pes77}, for a $C^{1+\alpha}$ diffeomorphism $f$, the Pesin unstable manifold of $x$ is defined as
$$W^u(x)=W^u_{\rm Pes}(x):=\{y\in M:~\limsup_{n\to\infty}\frac{1}{n}\log d(f^{-n}(x),f^{-n}(y))<0\},$$
and the Pesin stable manifold of $x$ is defined as
$$W^s(x)=W^s_{\rm Pes}(x):=\{y\in M:~\limsup_{n\to\infty}\frac{1}{n}\log d(f^{n}(x),f^{n}(y))<0\}.$$
It is known that $W^u_{\rm Pes}(x)$ is an immersed submanifold of dimension $\dim E^u(x)$, and $W^s_{\rm Pes}(x)$ is an immersed submanifold of dimension $\dim E^s(x)$. 
These manifolds exist on a set of full measure with respect to any hyperbolic ergodic measure.

We recall the definition of the homoclinic relation for ergodic hyperbolic measures (see \cite{BCS22}).
For two ergodic hyperbolic measures $\mu_1$ and $\mu_2$, we say that $\mu_1 \preceq \mu_2$, if  there are exist $\Lambda_1$ and $\Lambda_2$ such that 
\begin{itemize}
\item $\mu_1(\Lambda_1)>0$, $\mu_2(\Lambda_2)>0$;
\item for any $x\in\Lambda_1$ and $y\in\Lambda_2$, $W^u(x)$ intersects $W^s(y)$ transversely.
\end{itemize}
We say that $\mu_1$ is \textit{homoclinically related with} $\mu_2$, if $\mu_1 \preceq \mu_2$ and  $\mu_2 \preceq \mu_1$. 
Since the relation ``$\preceq$" is transitive ($\mu_1 \preceq \mu_2,\mu_2 \preceq \mu_3~\Rightarrow \mu_1 \preceq \mu_3$ ), the homoclinic relation is an equivalence relation for ergodic hyperbolic measures.

A special homoclinic relation occurs when a single ergodic measure is generated by a periodic orbit.
Suppose $O$ is a hyperbolic periodic orbit and $\mu_{O}$ is the ergodic measure supported on $O$. 
We say that an ergodic hyperbolic measures $\mu$ is homoclinically related with $O$, if $\mu$ is homoclinically related with $\mu_{O}$. 
In this case, for $\mu$-almost every $x$, we have $W^u(x)$ intersect $W^s(O)$ transversely, and $W^s(x)$ intersect $W^u(O)$ transversely.
Moreover, by the inclination lemma (see \cite[Lemma 2.7]{BCS22}), for $\mu$-almost every $x$, we have $W^s(O)$ \textit{accumulates to} $W^s(x)$ and $W^u(O)$ \textit{accumulates to} $W^u(x)$, i.e., 
\begin{itemize}
	\item for any compact discs $D^s\subset W^s(x)$ and $D^u\subset W^u(x)$, there exist $D^s_1,D^s_2,\cdots \subset W^s(O)$ and $D^u_1,D^u_2,\cdots \subset W^u(O)$ such that $D^s_n\rightarrow D^s$ and $D^u_n\rightarrow D^u$ in the $C^1$-topology. 
\end{itemize}
By the invariance of stable and unstable manifolds, this further implies
$${\rm Closure}(W^u({\rm Orb}(x)))={\rm Closure}(W^u(O)),~{\rm Closure}(W^s({\rm Orb}(x)))={\rm Closure}(W^s(O))$$
By \cite[Corollary 3.3]{BCS22},  two distinct ergodic measures of maximal entropy are not homoclinically related.
\begin{Lemma}[\cite{BCS22}, Corollary 3.3] \label{Lem: HES}
	Let $f: M \rightarrow M$ be a $C^{r},r>1$ surface diffeomorphism with positive entropy. 
	Then, for any ergodic measures $\mu$ and $\nu$ of maximal entropy, $\mu$ is homoclinically related with $\nu$ if and only of $\mu=\nu$.
\end{Lemma}

\subsection{Katok's shadowing and its extensions}\label{Sec:Katok-shadowing}
Katok's shadowding lemma \cite{Kat80} connects hyperbolic ergodic measures, hyperbolic periodic orbits, and hyperbolic basic sets.
For hyperbolic ergodic measures of surface diffeomorphisms, \cite{BCS22} provides a comprehensive summary of the extensions of Katok's shadowding lemma.
Below, we briefly outline the key points.

For a $C^{1+\alpha}$ surface diffeomorphism $f$, if $\mu$ is a hyperbolic ergodic measure, then there {are} a sequence of periodic orbits and a sequence of horseshoes\footnote{A horseshoe is an invariant, topologically transitive, infinite, totally disconnected, uniformly hyperbolic  compact set.} ``accumulating" to $\mu$. 

\paragraph{Related to periodic orbits.} There is a sequence of periodic orbits $\{O_n\}$ such that
\begin{itemize}
\item each $O_n$ is homoclinically related with $\mu$;
\item $\mu_{O_n}$ converges to $\mu$, $\lambda^{u}(\mu_{O_n},f)$ converges to $\lambda^{u}(\mu,f)$, $\lambda^{s}(\mu_{O_n},f)$ converges to $\lambda^{s}(\mu,f)$.
\end{itemize}

\paragraph{Related to horseshoes.}
There is an increasing sequence of horseshoes $\{\Lambda_n\}$ with
\begin{enumerate}
\item every ergodic measure supported on $\Lambda_n$ is homoclinically related with $\mu$;
\item if $O$ is a periodic orbit homoclinically to $\mu$, then there exists $n>0$ such that $O\subset \Lambda_{n}$;
\item $h_{\rm top}(f|_{\Lambda_n})$ converges to $h_\mu(f)$;
\item $\lambda^u(\Lambda_n,f)$ converges to $\lambda^u(\mu,f)$, $\lambda^s(\Lambda_n,f)$ converges to $\lambda^s(\mu,f)$ when $\dim M=2$. \label{Ie:2-4} 
\end{enumerate}
The Item \ref{Ie:2-4} can be seen as the main theorem in \cite{Gel16} and it also holds in any dimension, one can see \cite[Theorem 3.3]{ACW21}.

\subsection{Ledrappier-Young and Katok like entropy formula}\label{Sec:LY-Katok}
Let $f$ be a  diffeomorphism on a compact manifold $M$.  
For a compact subset $K\subset M$, $x\in K$, $n\in \mathbb N$ and $\varepsilon>0$, define $(n,\varepsilon)$-Bowen ball at $x$ with respect to $f$ by
$$B_n(x,\varepsilon,f;K):=\{y\in K:~d(f^i(y),f^i(x))<\varepsilon,~\forall 0\leq i<n\}.$$
We sometimes use the simplified notations when there is no confusion:
\begin{itemize}
\item When $K=M$, denote $B_n(x,\varepsilon,f)=B_n(x,\varepsilon,f;M)$;
\item $B_n(x,\varepsilon)=B_n(x,\varepsilon,f)$.
\end{itemize}

A subset $Y\subset K$ is said an $(n,\varepsilon)$-spanning set of $K$, if $K\subset \bigcup_{z\in Y}B_n(z,\varepsilon;K).$
We denote by $r_f(n,\varepsilon,K)$ the minimal cardinality of all possible $(n,\varepsilon)$-spanning sets of $K$.
The topological entropy of $K$ is thus defined to be 
$$h_{\rm top}(f,K):=\lim_{\varepsilon\to 0}\limsup_{n\to\infty}\frac{1}{n}\log r_f(n,\varepsilon,K).$$
Then, the topological entropy of $f$ satisfies $h_{\rm top}(f)=h_{\rm top}(f,M)$.

Let $\mu$ be a hyperbolic ergodic measure of $f$.
Recall the Oseledec theorem in Section \ref{Sec:LYE}.
For each $1\leq i \leq t$ such that $\lambda_i(\mu)>0$. We denote    
\begin{equation}\label{eq:SB}
	E^{u,i}(x)=\bigoplus_{j=1}^{i}E^j(x)
\end{equation}
By unstable manifolds theory \cite{Pes77}, when $f$ is $C^r$, for $\mu$-almost every $x$,
$$W^{u,i}(x):=\left\{y\in M: \limsup_{n\to\infty} \frac{1}{n} \log d(f^{-n}(x),f^{-n}(y))\leq -\lambda_i(\mu) \right\}$$
is a $C^{r}$ immersed submanifold tangent to $E^{u,i}(x)$ in \eqref{eq:SB} and inherits a Riemannian metric from $M$.
Denote this distance by $d_x^{u,i}$. 
With this distance we define $(n,\rho)$-Bowen balls along unstable manifolds by 
$$ V^{u,i}(x,n,\rho):=\left\{y\in W^{u,i}(x): d_{f^j(x)}^{u,i}(f^j(x),f^j(y))<\rho,~\forall 0\leq  j <n \right\}.$$

From \cite{LeS82} and \cite{LeY85}, there exists a measurable partition $\xi$ subordinate to $W^{u,i}$ with respect to $\mu$, meaning that for $\mu$-almost every point $x$, one has $\xi(x)\subset W^{u,i}(x)$ and $\xi(x)$ contains an open neighborhood of $x$ in $W^{u,i}(x)$.
From Rokhlin \cite{Rok67}, for the measurable partition $\xi$, there is a family of conditional measures $\{ \mu_{\xi(x)}\}$.

As in Ledrappier-Young \cite{LeY85}, the partial entropy along strong unstable manifolds $W^{u,i}$ is defined to be
\begin{equation} \label{eq:LY}
	h_{\mu}^{i}(f)=\lim_{\rho \rightarrow 0} \liminf_{n\rightarrow \infty} -\frac{1}{n} \log \mu_{\xi(x)}(V^{u,i}(x,n,\rho))=\lim_{\rho \rightarrow 0}\limsup_{n\rightarrow \infty} -\frac{1}{n} \log \mu_{\xi(x)}(V^{u,i}(x,n,\rho)).
\end{equation}
These limits exist and are constant $\mu$-almost everywhere. 
Note that if $\lambda^{i}(\mu)>0$ and $\lambda^{i+1}(\mu)\leq 0$, then $h_{\mu}^{i}(f)=h_{\mu}(f)$;
For each probability measure $\nu$ on $M$ and each $\delta>0$, we define 
$$r_{f}(n,\varepsilon,\nu,\delta):=\inf \{r_f(n,\varepsilon,Z):~ Z\subset M,~\nu(Z)>\delta\}.$$

\begin{Theorem}\label{Thm:Entropy-bound}
	Let $f,\mu,i,$ and $\xi$ be chosen as described above, then we have\footnote{Item (1) is not needed for the purposes of this paper.}
	\begin{enumerate}
		\item[(1)] 
		for $\mu$-almost every $x$ and every $\delta>0$ we have 
		$$h_{\mu}^{i}(f)\geq \lim_{\varepsilon\rightarrow 0}\limsup_{n\rightarrow +\infty} \frac{1}{n}\log r_{f}(n,\varepsilon, \mu_{\xi(x)},\delta),$$
		\item[(2)]  for every $\alpha>0$ and every $\delta_1,\delta_2\in (0,1)$, there exists a subset $K$ with $\mu(K)>1-\delta_1$, such that for every $x\in K$ and every $Z\subset \xi(x)$ with $\mu_{\xi(x)}(Z)>\delta_2$, one has
		$$h_{\mu}^{i}(f)\leq \lim_{\varepsilon\rightarrow 0}\liminf_{n\rightarrow +\infty} \frac{1}{n}\log r_{f}(n,\varepsilon, Z)+\alpha.$$
	\end{enumerate}
	
\end{Theorem}
\begin{proof}[Proof of Theorem \ref{Thm:Entropy-bound} (1)]
	Note that for $\mu$-almost every $x$ we have 
	\begin{equation}\label{eq:twoballs}
		\mu_{\xi(x)}(V^{u,i}(x,n,\rho))\leq \mu_{\xi(x)}(B_n(x,\rho)),~\forall \rho>0.
	\end{equation}
	By \eqref{eq:LY}, there exists $X$ with $\mu(X)=1$ such that for every $x\in X$,  one has
	$$h_{\mu}^{i}(f)\geq \lim_{\rho \rightarrow 0} \limsup_{n\rightarrow \infty} -\frac{1}{n} \log \mu_{\xi(x)}(B_n(x,\rho)).$$
	Choose $X'\subset X$ such that  $\mu(X')=1$ and  $\mu_{\xi(x)}(X)=1$ for every $x\in X'$. 
	Fix some $\alpha>0$ and $x\in X'$. 
	For $\varepsilon>0$ and $y\in \xi(x)\cap X$, since $\xi(x)=\xi(y)$, we have that
	$$\limsup_{n\rightarrow \infty} -\frac{1}{n} \log \mu_{\xi(x)}(B_n(y,\varepsilon))=\limsup_{n\rightarrow \infty} -\frac{1}{n} \log \mu_{\xi(y)}(B_n(y,\varepsilon))\leq h_{\mu}^{i}(f).$$
	For every $k>0$, consider
	$$A_{k}:=\left\{y\in  \xi(x)\cap X: \mu_{\xi(x)}(B_n(y,\varepsilon))\geq \exp (-n(h_{\mu}^{i}(f)-\alpha)),  ~\forall n\geq k \right\}.$$
	Then, we have $\mu_{\xi(x)}(A_k)\rightarrow 1$ as $k\rightarrow \infty$.
	
	For each $\delta>0$, choose $k$ large enough such that $\mu_{\xi(x)}(A_k)>\delta$.
	For $n>k$,  since for every $y\in A_k$ we have $\mu_{\xi(x)}(B_n(y,\rho))\geq \exp (-n(h_{\mu}^{i}(f)-\alpha))$, there are at most $\exp (n(h_{\mu}^{i}(f)+\alpha))$ numbers of disjoint $(n,\varepsilon)$-Bowen balls with center in $A_k$, and the same number of $(n,2\varepsilon)$-Bowen balls can cover $A_k$. Therefore, we have 
	$r_{f}(n,2\varepsilon,\mu_{\xi(x)},\delta)\leq \exp (n(h_{\mu}^{i}(f)+\alpha))$.
	
	Consequently, for every $x\in X'$ we have
	$$\lim_{\varepsilon \rightarrow 0}\limsup_{n\rightarrow +\infty} -\frac{1}{n}\log r_{f}(n,\varepsilon, \mu_{\xi(x)},\delta)\leq h_{\mu}^{i}(f)+\alpha.$$
	By the arbitrariness of $\alpha$, we complete the proof of Theorem \ref{Thm:Entropy-bound} (1).
\end{proof}

Since the reverse inequality of \eqref{eq:twoballs} may not hold, we introduce the following propositions, see \cite[Proposition 2.1, Proposition 2.2]{LY24}.
\begin{Proposition}\label{Prop:Two Balls}
	For every $\alpha>0$ and every $\delta\in(0,1)$, there exists $K\subset M$ with $\mu(K)>1-\delta$ and $\rho:=\rho_K>0$, such that for every $x\in K$, every measurable set $\Sigma\subset W^{u,i}_{{\rm loc}}(x)$ with $\mu_{\xi(x)}(\Sigma\cap K)>0$, and every finite partition $\cP$ with $ {\rm Diam}(\cP)<\rho$, one has
\begin{equation}\label{eq:LimK}
	h_{\mu}^{i}(f)\le\liminf_{n\to\infty}\frac{1}{n} \log \{P\in \cP^n: P\cap K \cap \Sigma \cap \xi(x)\neq \emptyset \}+\frac{\alpha}{2},
\end{equation}
where $\cP^n:=\bigvee_{j=0}^{n-1}f^{-j}\cP$.
\end{Proposition}

We now give the proof of Theorem \ref{Thm:Entropy-bound} (2).
\begin{proof}[Proof of Theorem \ref{Thm:Entropy-bound} (2)]
	For  $\alpha>0$ and $\delta_1,\delta_2>0$, let $\delta=\min\{\delta_1,\delta_2\}>0$. 
	We choose a subset $K_1$ and $\rho>0$ such that $\mu(K_1)>1-\frac{\delta^2}{2}$, and satisfying the conclusion of  Proposition \ref{Prop:Two Balls}.

	Choose a finite partition $\cP$ with $\mu(\partial \cP)=0$ and $\diam(\cP)\leq \rho$. 
	For $\varepsilon>0$ we define 
	$$U_{\varepsilon}=\{x: B(x,\varepsilon)~\text{is not contained in}~\cP(x)\}.$$ 
	Note that $U_{\varepsilon}$ is decreasing and $\mu(U_{\varepsilon})\rightarrow \mu(\partial \cP)=0$ whenever $\varepsilon\rightarrow 0$. 
	We choose $\alpha'>0$ such that 
	$$\sum_{j=0}^{\lceil n\alpha' \rceil} \binom{n}{j} \# \cP^{j}<e^{\frac{n\alpha}{2}}.$$
	Then, we can choose $\varepsilon>0$, a subset $K_2$ with $\mu(K_2)>1-\frac{\delta^2}{2}$ and $N_2\in \NN$, such that 
	$$\forall n\geq N_2,~\forall x\in K_2,~\# \{0\leq j<n: f^j(x)\in U_{\varepsilon}\}\leq n \alpha'.$$
	By the construction, for every $x\in K_2$, we have $B_n(x,\varepsilon)$ can be covered by at most $e^{\frac{n\alpha}{2}}$ numbers of elements which belong to $\cP^n$.
	
	Consider $K'=K_1\cap K_2$, one has $\mu(K')>1-\delta^2$. 
	Then, there exists a  subset $K \subset K'$ such that $\mu(K)>1-\delta$ and  $\mu_{\xi(x)}(K')>1-\delta$ for every $x \in  K$.	
	Now, fix some $x\in K$. For every subset $Z\subset \xi(x)$ with $\mu_{\xi(x)}(Z)>\delta$,
	let $\Sigma \subset Z \cap K'\subset K_1$ be a compact subset with $\mu_{\xi(x)}(\Sigma)>0$. 
	
	Let $\{y_i\}_{i=1}^{r_f(n,\varepsilon,\Sigma)}\subset \Sigma$ be an $(n,\varepsilon)$-spanning set of $\Sigma$. 
	Then, we have $\bigcup_{i=1}^{r_f(n,\varepsilon,\Sigma)} B_n(y_i,\varepsilon)\supset \Sigma$. 
	Hence,  by the construction of $K_2$, we have
	$$\{P\in \cP^n: P\cap K_1 \cap \Sigma \cap \xi(x)\neq \emptyset \}\leq e^{\frac{n\alpha}{2}} \cdot r_f(n,\varepsilon,\Sigma)\leq e^{\frac{n\alpha}{2}}  \cdot r_f(n,\varepsilon,Z).$$
	By Proposition \ref{Prop:Two Balls}, we have 
	$$h_{\mu}^{i}(f)\le \liminf_{n\to\infty}\frac{1}{n}\log  r_f(n,\varepsilon,Z)+\alpha.$$   	
	This completes the proof.
\end{proof}

\section{Burguet's and Buzzi-Crovisier-Sarig's results revisited} \label{Sec:revisited}

We state three theorems summarizing key the consequences we will need of some works of Burguet \cite{Bur24P} and Buzzi-Crovisier-Sarig \cite{BCS22}.

\subsection{Local finiteness}
\begin{Theorem}\label{Thm:BBCS-revisited}
Given a real number $r>1$, if a $C^r$ surface diffeomorphism $f$ satisfies
$$h_{\rm top}(f)>\frac{\lambda_{\min}(f)}{r}.$$
Then, there is a $C^r$ neighborhood $\mathcal U$ of $f$ such that for any $g\in\mathcal U$, one has that $1\le N_g<\infty$, i.e.,  $g$ has finitely many measures of maximal entropy. 
Moreover, distinct ergodic measures of maximal entropy of $g$ belong to disjoint homoclinic classes.
\end{Theorem}

\begin{proof}
Burguet \cite[Corollary 1]{Bur24P} has proved that if $h_{\rm top}(f)>\dfrac{\lambda^u(f)}{r}$, then $f$ admits an MME. 
Since MMEs for $f^{-1}$ are also MMEs for $f$, it follows directly from Burguet’s theorem that if
$$h_{\rm top}(f)=h_{\rm top}(f^{-1})>\frac{\lambda^u(f^{-1})}{r}=\frac{\lambda^s(f)}{r},$$
then $f^{-1}$, and hence $f$, admits an MME. 
This implies the condition in Burguet's result can be relaxed to be $h_{\rm top}(f)>\dfrac{\lambda_{\min}(f)}{r}$. 
Moreover, by the ``Main theorem revisited'' in Buzzi-Crovisier-Sarig \cite{BCS22}, it follows that $g$ has finitely many ergodic MMEs. 
By Theorem 2 of \cite{BCS22}, each homoclinic class of measures contains at most one MME (and distinct homoclinic classes are disjoint since they form a partition of the set of ergodic, hyperbolic measures).

To conclude the proof of the above theorem it suffices to show that this inequality is open. 
This follows immediately from the lower semicontinuity of the topological entropy (a classical result of Katok \cite{Kat80}) and the following easy remark.
\end{proof}

\begin{Lemma}\label{Lem:ups-growth-rate}
Both $\lambda^s(f)$ and $\lambda^u(f)$ vary upper semi-continuously with respect to $f$.
Hence, $\lambda_{\min}(f)$ varies upper semi-continuously with respect to $f$. 
\end{Lemma}

\subsection{Limits of high entropy measures}
From \cite[Corollary 1, Corollary2]{Bur24P}, we know the following theorem:

\begin{Theorem}\label{Thm:continuity-Lyapunov-exponent}
Given $r\ge 2$, if a $C^r$ surface diffeomorphism $f$ satisfies
$$h_{\rm top}(f)>\frac{\lambda_{\min}(f)}{r},$$
then for any sequence $\{f_m\}$ with $f_m \rightarrow f$ in the $C^r$ topology and any sequence of measures $ \{\mu_m\}$ satisfying
\begin{itemize}
\item for each $m$,  $\mu_m$ is an ergodic measure of $f_m$,
\item $\mu_m$ converges to $\mu$ and $h_{\mu_m}(f_m)$ converges to $h_{\rm top}(f)$;
\end{itemize}
we have $\mu$ is a measure of maximal entropy of $f$, and moreover, 
$$\lim_{m\to\infty}\lambda^+(\mu_m,f_m)=\lambda^+(\mu,f_m), ~\lim_{m\to\infty}\lambda^-(\mu_m,f_m)=\lambda^-(\mu,f_m).$$
\end{Theorem}
As mentioned earlier, one can consider $f^{-1}$ if necessary. 
Thus, one assumes that $\lambda_{\min}(f)=\lambda^u(f)$. 
Burguet \cite[Corollary 2]{Bur24P} has shown that $\mu$ is a measure of maximal entropy of $f$. 
Then from the proof of \cite[Corollary 1]{Bur24P}, it follows that $\lambda^+(\mu_m)$ converges to $\lambda^+(\mu)$. 
Using Equation~\eqref{e.integrate-Lyapunov} and the fact that the function $\log|{\rm Det}(Df)|$ is continuous, one has that 
$$\lim_{m\to\infty}(\lambda^-(\mu_m,f_m)+\lambda^+(\mu_m,f_m))=\lambda^-(\mu,f)+\lambda^+(\mu,f).$$
Consequently, $\lambda^-(\mu_m,f_m)$ converges to $\lambda^-(\mu,f)$.

\subsection{Burguet's reparametrization lemma}
We first recall the following fundamental notions and lemma from \cite{Bur24,Bur24P}.
Recall that a $C^r$ curve $\sigma:[-1,1]\rightarrow M$ is called \emph{bounded} if 
$$\sup_{2\leq s\leq r} \|D^s \sigma\|_{\sup} \leq \frac{1}{6} \|D\sigma\|_{\sup}.$$
For $\varepsilon>0$, a bounded $C^r$ curve $\sigma:[-1,1]\rightarrow M$ is called \emph{strongly $\varepsilon$-bounded}, if $\|D\sigma\|_{\sup}\leq \varepsilon$.

\begin{Lemma}[\cite{Bur24P}, Lemma 12]\label{Lem:local-reparametrization}
For $r\ge 2$, there is a constant $C_r>0$ with the following property.

Given any number $\Upsilon>0$, there is $\varepsilon(\Upsilon)>0$ such that if $g$ is a $C^r$ diffeomorphism with $\|g\|_{C^r}<\Upsilon$, then for any strongly $\varepsilon_{\Upsilon}$-bounded $C^r$ curve $\sigma:~[-1,1]\to M$ and any $\chi^+,\chi \in\mathbb Z$, there is a family of affine reparametrizations $\Theta$ such that
\begin{enumerate}
\item[(1)] $\{t\in[-1,1]:~x=\sigma(t),\lceil\log\|Dg(x)\|\rceil=\chi^+,~\lceil\log\|Dg|_{T_x{\sigma_*}}\|\rceil=\chi\}\subset \bigcup_{\theta\in\Theta}\theta([-1,1])$;
\item[(2)] $g\circ \sigma\circ\theta$ is bounded for any $\theta\in\Theta$;
\item[(3)] $\#\Theta\le C_r \exp(\frac{\chi^+-\chi}{r-1})$.
\end{enumerate}
\end{Lemma}

We now apply inductively the above reparametrization to control the iterates of a small curves staying near a given orbit.
Let $r(M)$ be a positive constant such that $\exp_x$ maps $T_xM(2r(M))$ diffeomorphically onto a neighborhood of $B(x,r(M))$ for every $x\in M$.

Denote $\sigma_*:={\rm Image}(\sigma)$.
For $q>0$ and for $m$-diffeomorphisms $(g^{i_j})_{j=0}^{m-1}$ with $1\leq i_j \leq q$ for all $0\leq j<m$, and $2m$-integers $(\chi^+_j,\chi_j)_{j=0}^{m-1}$,  we define the set
$$\Sigma(\{\chi^{+}_k, \chi_k, g^{i_k}\}_{k=0}^{m-1})\subset \sigma_*$$
as follows: $z\in \Sigma(\{\chi^{+}_k, \chi_k, g^{i_k}\}_{k=0}^{m-1})$ if and only if for every $0\le k \le m-1$
$$\lceil\log\|Dg^{i_k}(g^{i_0+\cdots+i_{k-1}}(z))\|\rceil=\chi^+_j,~\lceil\log\|Dg^{i_k}|_{T_{g^{i_0+\cdots+i_{k-1}}(z)}{(g^{i_0+\cdots+i_{k-1}}(\sigma_*))}}\|\rceil=\chi_j$$
Meanwhile, for any $z\in \sigma_*$, one defines
$$\chi^+_j(z):=\lceil\log\|Dg^{i_k}(g^{i_0+\cdots+i_{k-1}}(z))\|\rceil,~~\chi_j(z):=\lceil\log\|Dg^{i_k}|_{T_{g^{i_0+\cdots+i_{k-1}}(z)}{(g^{i_0+\cdots+i_{k-1}}(\sigma_*))}}\|\rceil.$$
Let $\Upsilon>\max\{\|g\|_{C^r},\cdots,\|g^q\|_{C^r}\}$ and $0<\varepsilon<\min\{\frac{\Upsilon}{10},\frac{r(M)}{10\Upsilon}\}$.
We denote $B(x,\varepsilon,\{g^{i_k}\}_{k=0}^{m-1})$ by 
$$B(x,\varepsilon,\{g^{i_k}\}_{k=0}^{m-1}):=\{y\in M: d(x,y)\leq \varepsilon,~d(g^{i_0+\cdots+i_{k-1}}(x),g^{i_0+\cdots+i_{k-1}}(y))\leq \varepsilon, \forall 0\leq k<m \}.$$
The following lemma is close to Burguet \cite[Lemma 8]{Bur24}.
\begin{Lemma}\label{Lem:reparametrization-up-to-m}
Given $2m$-integers $(\chi^+_j,\chi_j)_{j=0}^{m-1}$ and $m$-diffeomorphisms $(g^{i_j})_{j=0}^{m}$, for each $x\in \sigma_*$, there exists a family $\Gamma_m'=\Gamma_m'(\Sigma(\{\chi^{+}_k, \chi_k, g^{i_k}\}_{k=0}^{m-1}),x)$ of reparametrizations satisfying
$$\#\Gamma_m'\le C_r^m\exp\left(\sum_{i=0}^{m-1}\frac{\chi_i^+-\chi_i}{r-1}\right),$$
and with the following properties
\begin{enumerate}
\item[(1)] $B(x,\varepsilon,\{g^{i_k}\}_{k=0}^{m-1})\cap \Sigma(\{\chi^{+}_k, \chi_k, g^{i_k}\}_{k=0}^{m-1})\subset\bigcup_{\gamma\in\Gamma_m'}\sigma\circ\gamma([-1,1])$
\item[(2)] for every $\gamma\in\Gamma_m'$, $g^{j}\circ\sigma\circ\gamma$ is strongly $2\varepsilon$-bounded for any $j\in \{0,i_0,i_0+i_1,\cdots,i_0+\cdots+i_{m-1}\}$.
\end{enumerate}
\end{Lemma}
\begin{proof}
Since the intersection of a bounded curve of length smaller than the injectivity radius of the exponential map with an $\varepsilon$-ball is contained in a strongly $2\varepsilon$-bounded curve,
the proof of this lemma follows by applying the reparametrization lemma (see Lemma \ref{Lem:local-reparametrization}) inductively to the sequence $\{g^{i_k}\}_{k=0}^{m-1}$.
\end{proof}

\subsection{Sard's theorem for dynamical foliations}
We recall basic facts about uniformly hyperbolic sets including estimates on the transverse dimension of the associated foliation, and state a consequence of dynamical Sard's lemma from \cite{BCS22}.

\medbreak
 
A compact invariant set $\Lambda$ is said to be \emph{hyperbolic} if there are constants $C>0$ and $\lambda\in(0,1)$, and a $Df$-invariant continuous splitting $T_\Lambda M=E^s\oplus E^u$, such that for any $x\in\Lambda$ and any $n\in\mathbb N$, one has
$$\|Df^n|_{E^s(x)}\|\le C\lambda^n,~~~\|Df^{-n}|_{E^u(x)}\|\le C\lambda^n.$$
Recall that a \emph{horseshoe} is a hyperbolic invariant compact set $\Lambda$ which is infinite, totally discontinuous and topologically transitive (there is $x\in\Lambda$ whose orbit is dense in $\Lambda$). 
See also Section~\ref{Sec:Katok-shadowing}.

\begin{Theorem}\label{Thm:hyperbolic-easy}
Any hyperbolic invariant compact set $\Lambda$ admits measures of maximal entropy, i.e., there is an ergodic measure $\mu$ supported on $\Lambda$ such that $h_\mu(f)=h_{\rm top}(f|_\Lambda)$.
 
Any hyperbolic set $\Lambda$ admits a continuation, i.e., there is a $C^1$ neighborhood $\mathcal U$ of $f$ such that any $g\in\mathcal U$ admits a hyperbolic set $\Lambda_g$ which is topological conjugate to $\Lambda$. 
In particular, one has that $h_{\rm top}(g|_{\Lambda_g})=h_{\rm top}(f|_{\Lambda})$.
 
If $\Lambda$ is a horseshoe, then for any two points $x,y\in\Lambda$, there is an integer $n\ge0$ such that $W^u(x)$ is $C^1$-accumulated by $W^u(f^ny)$ in the compact open $C^1$-topology (recall the notion $W^u(x)$ being accumulated by $W^u(f^ny)$ in Section \ref{Sec:homoclinic-relation}).
\end{Theorem}

We will use the dynamical Sard's theorem from Buzzi-Crovisier-Sarig \cite[Theorem 4.2]{BCS22}. 
Based on Section 4 of \cite{BCS22}, we will state its version for a {horseshoe} of a $C^r$ surface diffeomorphism. 
Note that \cite{BCS22} considered some continuous lamination. 
However, in their applications, this continuous lamination is taken to be $W^u(\Lambda)$.
 
We need the following simple topological notion.
Let us say that a curve $A$ \emph{topologically crosses} a curve $B$ if there is a topological disk $\Delta$ such that $\Delta\setminus B$ has exactly two connected components $\Delta_-,\Delta_+$ and some connected component of $A\cap \Delta$ intersects both $\Delta_-,\Delta_+$. 
See Figure~\ref{Fig:TC} (i) and (ii). 
Recall that $\lambda^s(\Lambda,f)$ was defined in \eqref{eq:Lam}.
 
\begin{Theorem}\label{Thm:basic-set-sard}
Let $f$ be a $C^r$ surface diffeomorphism. 
Assume that $\Lambda$ is a {horseshoe} satisfying 
$$\frac{h_{\rm top}(f|_\Lambda)}{\lambda^s(\Lambda,f)}>1/r,$$
and that $\gamma$ is a $C^r$ curve. 
If for all points $x\in\Lambda$, $\gamma$ topologically crosses $W^u(f^kx)$ for some $k\ge0$, then for all points $y\in\Lambda$, there is an integer $k'\ge0$ such that $\gamma\pitchfork W^u(f^{k'}y)\neq\emptyset$.
\end{Theorem}

\begin{figure}[htp]
	\centering
	\includegraphics[width=0.6\linewidth]{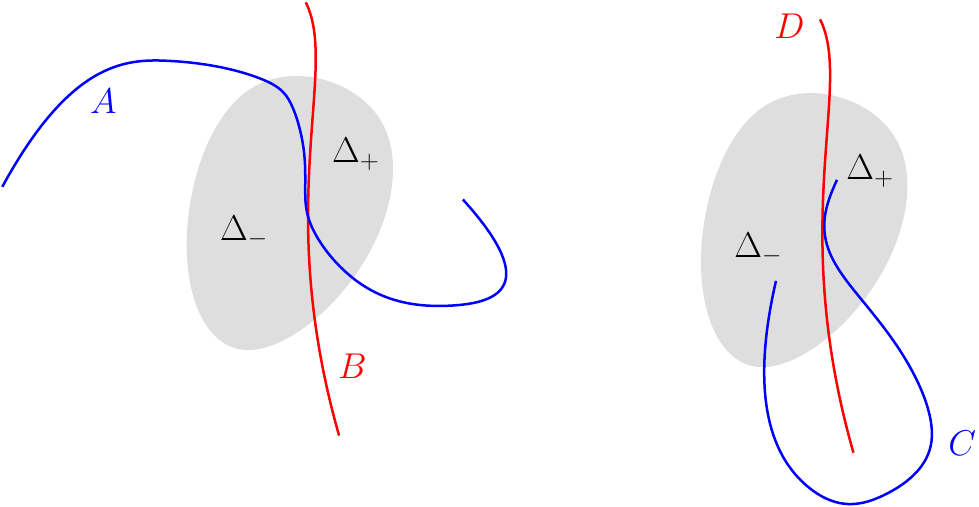}
	\caption{From left to right: (i) a curve $A$ topologically crosses a curve $B$; (ii) a curve $C$ does not topologically cross a curve $D$. }
	\label{Fig:TC}
\end{figure}

\section{The main reduction theorem - setup }\label{Sec:MRT}

We consider the following version of Theorem~\ref{Thm:HomoRel}.
It generalizes \cite[Proposition 5.4]{BCS22} to the large entropy surface setting. 
It is the key technical result underlying the proof of all main results through Corollary~\ref{Cor:LimitDisjoint}, see Section~\ref{Sec:ProofMainThms}.

\begin{Theorem}\label{Thm:main-reduction}
Let $r\ge 2$ and  $f$ be a $C^r$ surface diffeomorphism  satisfying
$$h_{\rm top}(f)>\frac{\lambda_{\min}(f)}{r}.$$
Assume that $\mu_0$ is a measure of maximal entropy of $f$, and there exist $\alpha\in(0,1]$ and an ergodic measure of maximal entropy $\nu$ such that
$$\mu_0=\alpha\nu+(1-\alpha)\nu'.$$
Let $O$ be a hyperbolic periodic orbit that is homoclinically related with $\nu$.

\smallskip

Then, there exists a $C^r$ neighborhood $\mathcal U$ of $f$, a neighborhood $U$ of $\mu_0$, and a number $h\in(0,h_{\rm top}(f))$, such that for any $g\in \mathcal U$, and any ergodic measure $\mu\in U$ of $g$ with $h_\mu(g)>h$, we have that $\mu$ is homoclinically related with the hyperbolic continuation $O_g$ of the orbit $O$.
\end{Theorem}

\begin{Remark}
The proof of Theorem~\ref{Thm:main-reduction} differs from \cite[Proposition 5.4]{BCS22} in two key respects. 
First we need to use the local entropy control from Burguet's reparametrization lemma, see Lemma~\ref{Lem:reparametrization-up-to-m} below. Second, our assumption $\htop(f)>\lambda_{\min}(f)/r$ only gives that one of the stable or unstable laminations has a large transverse dimension. 
Hence we must use non-symmetrical arguments unlike in the proof of \cite[Proposition 5.4]{BCS22} (but similarly as in \cite[Section 6]{BCS22}, see, e.g., \cite[Proposition 6.4]{BCS22}).
\end{Remark}

In the rest of Section~\ref{Sec:MRT}, we consider $\mu$ as above and select $\mu$-typical points in some local unstable manifold so that the entropy of $\mu$ can be bounded in terms of the number of reparametrizations needed to cover these typical points. 
This will involve consequences of Burguet's \cite{Bur24P}: (i) the convergence of the exponents near MMEs; (ii) the reparametrization lemma for curves; as well as a more classical result: (iii) entropy formula of Ledrappier-Young. 
The proof of Theorem~\ref{Thm:main-reduction} will be concluded in Section~\ref{Sec:MRT-geo}.

\subsection{Choose constants and neighborhoods of $f$}

\subsubsection{Choose $\delta>0$ and $h_1>0$}\label{SEC:delta}
We take $\delta>0$ such that
$$0<\delta<\frac{\alpha}{20} h_{\rm top}(f). $$
We then choose $h_1>0$ as follows.
\begin{Lemma}\label{Lem:h-neighborhood}
There exist a $C^r$ neighborhood $\mathcal U_{\rm ent}$ of $f$ and a number $h_1$ such that for any $g\in\mathcal U_{\rm ent}$, we have
\begin{equation}\label{eq:deltah1}
   \max\left(\frac{\lambda_{\min}(g)}{r},4\delta+(1-\frac{\alpha}{2})h_{\rm top}(f)\right)< h_1 < \htop(f).
\end{equation}
\end{Lemma}
\begin{proof}
Note $4\delta+(1-\frac{\alpha}{2})h_{\rm top}(f)< (1-(3/10)\alpha)\htop(f)<\htop(f)$. 
Recall also that $f$ has large entropy. Thus, we have that 
         $$\max\left((\lambda_{\min}(f)+\eta)/r,4\delta+(1-\frac{\alpha}{2})h_{\rm top}(f)\right)< \htop(f).$$
for any $\eta\in\left(0,r\cdot h_{\rm top}(f)-\lambda_{\min}(f)\right)$. 
Lemma~\ref{Lem:ups-growth-rate} implies that  $\lambda_{\min}(g)<\lambda_{\min}(f)+\eta$ for all $g$ in some neighborhood $\mathcal U_{\rm ent}$ of $f$, hence
  \begin{align*}
    \sup_{g\in\mathcal U_{\rm ent}} \max&\left((\lambda_{\min}(g)/r,4\delta+(1-\frac{\alpha}{2})h_{\rm top}(f)\right)\\ 
                    &\le \max\left((\lambda_{\min}(f)+\eta)/r,4\delta+(1-\frac{\alpha}{2})h_{\rm top}(f)\right)< h_{\rm top}(f).
  \end{align*}
We can pick $h_1$ as claimed.
\end{proof}
\subsubsection{Choose $h>0$ and the integer $q$}

\begin{Lemma}\label{Lem:Lyapunov-quantity}

Assume that $f$ is a $C^r$ surface diffeomorphism satisfying $h_{\rm top}(f)>\frac{\lambda_{\min}(f)}{r}$. 
Given $\eta>0$, there exists a positive integer $q_f\in\mathbb N$ with the following properties.

For any $q\ge q_f$, there exist a $C^r$ neighborhood $\mathcal U_{\rm con}$ of $f$ and a number $0<h<h_{\rm top}(f)$, such that for every $g\in \mathcal U_{\rm con}$ and every ergodic measure $\mu$ of $g$ with $h_{\mu}(g)>h$, the following holds:
\begin{align*}
	\frac{1}{q}\int\log\|D_xg^q\|{\rm d}\mu(x)&\in [\lambda^+(\mu,g),\lambda^+(\mu,g)+\eta);\\
	-\frac{1}{q}\int\log \|D_xg^{-q}\|{\rm d}\mu(x) &\in (\lambda^-(\mu,g)-\eta,\lambda^-(\mu,g)].
\end{align*}

\end{Lemma}
\begin{proof}
We will prove the result for the upper Lyapunov exponents. 
The result for the lower Lyapunov exponents can then be obtained by considering $f^{-1}$. 

Fix $\eta>0$.  
Recall that $N_f$ is the number of ergodic measures of maximal entropy of $f$. 
Denote by $\{\mu_1,\cdots,\mu_{N_f}\}$ the set of ergodic measures of maximal entropy of $f$. 
For each $i$, there exists $q_i\in\mathbb N$ such that for any $q\ge q_i$, one has that
$$\frac{1}{q}\int\log\|Df^q\|{\rm d}\mu_i \in \left[\lambda^+(\mu_i),\lambda^+(\mu_i)+\frac{\eta}{3}\right).$$
We define
$$q_f:=\max\{q_1,q_2,\cdots,q_{N_f}\}.$$
For any measure of maximal entropy $\nu$ of $f$, one has the decomposition $\nu=\sum_{i=1}^{N_f}\alpha_i\mu_i$ with $\alpha_i\ge 0$ and $\sum_{i=1}^{N_f}\alpha_i=1$. Thus, we get
$$\lambda^+(\nu)=\sum_{i=1}^{N_f}\alpha_i \lambda^+(\mu_i),$$
and for every $q\ge q_f$
$$\frac{1}{q}\int\log\|Df^q\|{\rm d}\nu= \sum_{i=1}^{N_f}\alpha_i \frac{1}{q}\int\log\|Df^q\|{\rm d}\mu_i.$$
Thus, one has
$$\frac{1}{q}\int\log\|Df^q\|{\rm d}\nu\in\left[\lambda^+(\nu),\lambda^+(\nu)+\frac{\eta}{3}\right).$$
Choose a $C^r$ neighborhood $\mathcal U^1_{\rm con}$ of $f$ and a constant $\delta_{\rm con}>0$, such that for any $g\in \mathcal U^1_{\rm con}$, and any probability measures $\mu$ and $\nu$ with $d(\mu,\nu)<\delta_{\rm con}$, we have
$$\left|\frac{1}{q}\int\log\|Df^q\|{\rm d}\nu-\frac{1}{q}\int\log\|Dg^q\|{\rm d}\mu\right|<\frac{\eta}{3}.$$

By Theorem \ref{Thm:continuity-Lyapunov-exponent}, we can choose a $C^r$ neighborhood $\mathcal U^2_{\rm con}\subset \mathcal U^1_{\rm con}$ of $f$ and a positive number $h_1<h<h_{\rm top}(f)$ such that for every $g\in \mathcal U^2_{\rm con}$ and every ergodic measure $\mu$ of $g$ with $h_{\mu}(g)>h$,  there exists a measure of maximal entropy $\nu$ of $f$ such that
\begin{itemize}
\item $d(\mu,\nu)\leq \delta_{\rm con}$, $|\lambda^+(\mu,g)-\lambda^+(\nu,f)|<\frac{\eta}{3}$ and $|\lambda^{-}(\mu,g)-\lambda^{-}(\nu,f)|<\frac{\eta}{3}$.
\end{itemize}
Thus, we have 
$$\frac{1}{q}\int\log\|Dg^q\|{\rm d}\mu<\frac{1}{q}\int\log\|Df^q\|{\rm d}\nu+\frac{\eta}{3} \le\lambda^+(\nu,f)+\frac{2\eta}{3}\le \lambda^+(\mu,g)+\eta.$$
Clearly, by the property of Lyapunov exponents, we also have
$$\frac{1}{q}\int\log\|Dg^q\|{\rm d}\mu\ge \lambda^+(\mu,g).$$
Finally, by the lower semi-continuity of topological entropy at $f$ (a classical result of Katok \cite{Kat80}), we can choose a $C^r$ neighborhood $\mathcal U_{\rm con}\subset \mathcal U^2_{\rm con}$ of $f$ such that $h_{\rm top}(g)>h$ for every $g\in \mathcal U_{\rm con}$. 
Thus, the conclusion of the lemma holds for $q_f$, $\mathcal U_{\rm con}$ and $h_2$.
\end{proof}

Based on Lemma \ref{Lem:Lyapunov-quantity}, we can choose $q\in \NN$, $h>0$ and a $C^r$ neighborhood  $\mathcal U_{\rm con}$ of $f$, such that for any $g\in \mathcal U_{\rm con}$ and any ergodic measure $\mu$ of $g$ with $h_{\mu}(g)>h$ (such measures exist), the following conditions hold (the constant $C_r$ will be chosen according to Lemma \ref{Lem:local-reparametrization})
\begin{align}
\frac{1}{r-1}\left(\frac{1}{q}\int\log\|Dg^q\|{\rm d}\mu-\lambda^+(\mu,g)\right)&<\frac{\delta}{5}; \label{e.q-Lyapunov} \\
\frac{1}{r-1}\left(\frac{1}{q}\int\log\|Dg^{-q}\|{\rm d}\mu-\lambda^+(\mu,g^{-1})\right)&<\frac{\delta}{5}; \label{e.q-Lyapunov-s}\\
 \frac{1}{q}\log\left(q\cdot(\|Df\|_{\rm sup}+\|Df^{-1}\|_{\rm sup}+4)\right)&<\frac{\delta}{10}; \label{e.q-norm} \\
\max\left\{ \frac{1}{q(r-1)},\frac{\log C_r}{q} \right\}&<\frac{\delta}{10}. \label{e.q-Cr} 
\end{align}

\subsubsection{Choose $\varepsilon>0$}

We will choose $\varepsilon>0$ according to Burguet's reparametrization lemma recalled above as Lemma~\ref{Lem:local-reparametrization}.

For the $C^r$ diffeomorphism $f$, we take $\Upsilon_f>0$ such that
$$\max_{1\leq j\leq q}2\|f^{j}\|_{C^{r}}<\Upsilon_f.$$ 
Then, there exists a $C^r$ neighborhood $\mathcal U_{\rm boun}$ such that for any $g\in\mathcal U_{\rm boun}$, the following holds:
$$\max_{1\leq j\leq q} \|g^j\|_{C^{r}}<\Upsilon_f,~\log\|Dg\|_{\sup}\leq \log\|Df\|_{\sup}+1,~\log\|Dg^{-1}\|_{\sup}\leq \log\|Df^{-1}\|_{\sup}+1.$$
Recall the numbers $C_r>0$ and $\varepsilon(\Upsilon_f)>0$ defined by Burguet's Lemma~\ref{Lem:local-reparametrization}.
We take $\varepsilon>0$  and $N_0\in\mathbb N$ such that (where $r(M)$ is a positive constant such that $\exp_x$ maps $T_xM(2r(M))$ diffeomorphically onto a neighborhood of $B(x,r(M))$ for every $x\in M$)
\begin{equation}\label{e.bound-varepsilon}
	0<\varepsilon<\frac{\varepsilon(\Upsilon)}{10},~~\max_{1\le j\le q}\{\|Df^j\|_{\rm sup}\}\cdot \varepsilon<\frac{r(M)}{10},~~r_f(\frac{\varepsilon}{2},N_0)<\exp\big((h_{\rm top}(f)+\delta)N_0\big).
\end{equation}
By reducing $\mathcal U_{\rm boun}$ if necessary, we can assume that for any $g\in \mathcal U_{\rm boun}$
\begin{equation}\label{e.bound-U}
	r_g(\varepsilon,N_0)<\exp\big((h_{\rm top}(f)+\delta)N_0\big)
\end{equation}
We take $C_0\in\mathbb N$ such that
\begin{equation}\label{eq:C0}
	C_0>\frac{2 h_{\rm top}(f)}{\delta}N_0.
\end{equation}
We denote $\cU'=\cU_{\rm ent}\bigcap \cU_{\rm con} \bigcap \cU_{\rm boun}$.
Recall the choice of $h$ implies that
\begin{equation}\label{eq:hU'}
	\forall g\in \cU',~h\in \left(
	\frac{\lambda_{\min}(g)}{r},h_{\rm top}(g)\right)\cap \left(4\delta+(1-\frac{\alpha}{2})h_{\rm top}(f), h_{\rm top}(f)\right).
\end{equation}

\subsection{Choose a typical set $K$}\label{Sec:choose-several}
Fix  $g\in \cU'$ and let $\mu$ be an ergodic measure of $g$ with $h_{\mu}(g)=h_{\mu}(g^{-1})>h$. 
Let $\xi$ be a measurable partition subordinate to $W^{u}$, with respect to the measure $\mu$, and let $\{\mu_{\xi(x)}\}$ denote the family of conditional measures for $\xi$.
Similarly, by considering $g^{-1}$, let $\xi^{s}$ be a measurable partition subordinate to $W^{s}$ of $\mu$, and let $\{\mu_{\xi^s(x)}\}$ denote the family of conditional measures for $\xi^s$.  
Recall that the $\delta$ was chosen in Section \ref{SEC:delta}.

We choose a subset $K:=K(g,\mu)$ that satisfies the following properties:
\begin{enumerate}
\item[(1)] $\mu(K)>1-\frac{\alpha}{2}$ and $\frac{1}{n}\sum_{j=0}^{n-1}\delta_{g^j(x)}\to\mu$ uniformly for $x\in K$;
\item[(2)] as $n\rightarrow \pm \infty$, $\frac{1}{n}\log \|Dg^n|_{E^u(x)}\| $ and $\frac{1}{n}\log \|Dg^n|_{E^s(x)}\| $ converge uniformly to $\lambda^{u}(\mu,g)$ and  $\lambda^{s}(\mu,g)$ on $K$.
\item[(3)]  for every $x\in K$, there are $C^r$ strongly $\varepsilon$-bounded curves $\sigma:[-1,1]\rightarrow M$ and $\sigma^s:[-1,1]\rightarrow M$ such that $\xi(x)\subset {\rm Image}(\sigma)\subset W^u_{\loc}(x)$ and $\xi^s(x)\subset {\rm Image}(\sigma^s)\subset W^s_{\loc}(x)$;
\item[(4)] for every $x\in K$, for every subset $Z$ with $\mu_{\xi(x)}(W^u_{\loc}(x)\cap Z)>0$, one has
$$h_{\mu}(g)\leq \lim_{\rho\rightarrow 0}\liminf_{n\rightarrow +\infty} \frac{1}{n}\log r_{g}(n,\rho, W^u_{\loc}(x)\cap Z)+\delta;$$
\item[(5)] for every $x\in K$, for every subset $Z$ with $\mu_{\xi^s(x)}(W^s_{\loc}(x)\cap Z)>0$, one has
$$h_{\mu}(g)=h_{\mu}(g^{-1})\leq \lim_{\rho\rightarrow 0}\liminf_{n\rightarrow +\infty} \frac{1}{n}\log r_{g^{-1}}(n,\rho, W^s_{\loc}(x)\cap Z)+\delta.$$
\item[(6)] for every $x\in K$, there exists a sequences of hyperbolic periodic points $\{p_n\}$ such that each $p_n$ is contained in a hyperbolic basic set $\Lambda_n$ satisfying the following properties:
\begin{itemize}
\item $p_n$ converges to $x$, and for each $n$, the orbit of $p_n$ is homoclinically related with $\mu$;
\item the entropy and Lyapunov exponents of  $\Lambda$  converge to those of  $\mu$:
$$\lim_{n\to\infty}h_{\rm top}(g|_{\Lambda_n})=h_\mu(g),~ \lim_{n\to\infty}\lambda^s(\Lambda_n,g)=\lambda^s(\mu,g),~\lim_{n\to\infty}\lambda^u(\Lambda_n,g)=\lambda^u(\mu,g);$$
\item for every $n$, the unstable manifolds of the orbit of $p_n$ accumulate to $W^u({\rm Orb}(x))$:
 $${\rm Closure}(W^u({\rm Orb}(x)))\subset {\rm Closure}(W^u({\rm Orb}(p_n)));$$
\item for every $n$, the stable manifolds of the orbit of $p_n$ accumulate to $W^s({\rm Orb}(x))$:
$${\rm Closure}(W^s({\rm Orb}(x)))\subset {\rm Closure}(W^s({\rm Orb}(p_n)));$$
\end{itemize}
\end{enumerate}
 \begin{Remark}
 	Item (1), Item (2), and Item (3) can be realized on a set of arbitrarily large $\mu$-measure.
 	Item (4) and Item (5) follow directly from Theorem \ref{Thm:Entropy-bound} by considering $g$ and $g^{-1}$.
 	Item (6) is based on katok shadowing lemma and its extensions (see Section \ref{Sec:Katok-shadowing}). 
 	We can assume that  $W^u(x)$ intersect $W^s({\rm Orb}(p_n))$ transversely, and $W^s(x)$ intersect $W^u({\rm Orb}(p_n))$ transversely for every $x\in K$ and every $n>0$, then the third and fourth statements of Item (6) follow.
 \end{Remark}
 
\subsection{Bounding the entropy of $\mu$ at a fixed scale}

We fix a point $x_0\in K$, and consider $\sigma:[-1,1]\rightarrow M$ be a strongly $\varepsilon$-bounded curve with $\xi(x_0)\subset \sigma_*\subset W^u_{\loc}(x_0)$. 
We apply Burguet's reparametrization lemma (more precisely, our Lemma~\ref{Lem:reparametrization-up-to-m}) using the ergodic properties ensured by the choice of $K$.

\begin{Proposition}\label{Pro:bound-reparametrizations}
	Let $\sigma$ be chosen as above, for every $n>0$ and every $x\in \sigma_*$, there is a family of reparametrizations $\Gamma_n=\Gamma_n(B_n(x,\varepsilon,g)\cap W^u_{\rm loc}(x_0)\cap K)$ such that	
	\begin{enumerate}
		\item[(1)] for every $\gamma\in \Gamma_n$, one has that $\|g^i\circ\sigma\circ\gamma\|\le1$ for any $1 \le i\le n$;
		
		\item[(2)] $\sigma^{-1}(K\cap W^u_{\rm loc}(x_0)\cap B_n(x,\varepsilon,g))\subset\bigcup_{\gamma\in\Gamma_n}\gamma([-1,1])$;
		
		\item[(3)] $\limsup\limits_{n\to\infty}\dfrac{1}{n}\log \#\Gamma_n\leq\delta$ (recall that the $\delta$ was chosen in Section \ref{SEC:delta}).
	\end{enumerate}
\end{Proposition}
\begin{proof}
	By the choice of $K$, there exists $n_K$ such that for every $n\geq n_K$ and every $x\in K$
	\begin{align}
	\frac{1}{n}\sum_{j=0}^{n-1}\log\|Dg^q(g^{j}(x))\|&\leq \int \log\|Dg^q\|{\rm d}\mu+\frac{\delta(r-1)}{5};\\
	\frac{1}{n}\log\|Dg^n|E^u(x)\|&\geq  \lambda^u(\mu,g)-\frac{\delta(r-1)}{5}. \label{eq:Ly1}
	\end{align}
	Consider $n=mq+p$ with $mq>n_K$ and $0\leq p <q$. Then, for $x\in K$, there exists $0\leq c_n(x)<q$ such that 
	\begin{equation}\label{eq:Ly2}
		\frac{1}{n}\sum_{j=0}^{m-1} \log\|Dg^q(g^{qj+c_n(x)}(x))\| \leq \frac{1}{q}\int \log\|Dg^q\|{\rm d}\mu+\frac{\delta(r-1)}{5}.
	\end{equation}
	 Fix $0\leq c<q$, for each $x\in K$ with $c_n(x)=c$, we consider 
	 \begin{align*}
   &\chi^+_{c}(x):=\lceil\log\|Dg^{c}(x)\|\rceil,~\chi_{c}(x):=\lceil \log\|Dg^{c}|_{T_{x} \sigma_*}\|\rceil,~c>0;~~\chi^+_{c}(x)=\chi_{c}(x)=0,~c=0;\\
   &\chi^+_{c,j}(x):=\lceil\log\|Dg^{q}(g^{qj+c}(x))\|\rceil,~\chi_{c,j}(x):=\lceil \log\|Dg^{q}(g^{qj+c}(x))|_{T_{g^{qj+c}(x)} (g^{qj+c}\sigma)_*} \|\rceil, 0 \leq  j< m.
	 \end{align*}
	 Fix $(m+1)$-diffeomorphisms $(g^{c},g^q,\cdots,g^q)$ and $2(m+1)$-integers 
	 $$(\chi^+_{c},\chi_{c},\chi^+_{c,0},\chi_{c,0},\cdots,\chi^+_{c,m-1},\chi_{c,m-1}).$$
	 Note that for every $0\leq c<q$
	 $$\max \{\chi^+_{c}(z),\chi_{c}(z)\}\leq c(\log\|Dg\|+\log \|Dg^{-1}\|)+2\leq q(\log\|Dg\|+\log \|Dg^{-1}\|)+2$$
	 and for every $0\leq j<m$
	 $$\max \{\chi^+_{c,j}(z),\chi_{c,j}(z)\}\leq q(\log\|Dg\|+\log \|Dg^{-1}\|)+2.$$
	 Hence, we have
	 \begin{align*}
	 	&\# \{(\chi^+_{c},\chi_{c},\chi^+_{c,0},\chi_{c,0},\cdots,\chi^+_{c,m-1},\chi_{c,m-1}):~ \exists z\in K,~\text{such that}~c_n(x)=c, \\
	 	&~~~~~~~~~~ \chi^+_{c}(z)=\chi^+_{c},\chi_{c}(z)=\chi_{c},~\chi^+_{c,j}(z)=\chi^+_{c,j},\chi_{c,j}(z)=\chi_{c,j},~\forall 0\leq j<k\} \\
	 	  \leq   &q\cdot \left(q(\log\|Dg\|+\log \|Dg^{-1}\|)+2\right)^{2m+2} \\
	 	  \leq   &q\cdot \left(q(\log\|Df\|+\log \|Df^{-1}\|+4)\right)^{2m+2}.
	 \end{align*}
{Given integer indices $c,\chi_c,\chi_c^+,\chi_{c,j},\chi_{c,j}^+$, we introduce}
	 $$\Sigma:=\{z\in K: c_n(x)=c,~\chi^+_{c}(z)=\chi^+_{c},\chi_{c}(z)=\chi_{c},~\chi^+_{c,j}(z)=\chi^+_{c,j},\chi_{c,j}(z)=\chi_{c,j},~\forall 0\leq j<k \}.$$
{Note that we omit the indices from the notation $\Sigma$.}
 Therefore, by Lemma \ref{Lem:reparametrization-up-to-m}, there exists a family of reparametrizations $\Gamma'(\Sigma)$ such that 
    \begin{itemize}
    	\item $\# \Gamma'(\Sigma) \leq C_r^{m+1} \exp \left( \frac{1}{r-1}(\chi_{c}^{+}-\chi_c+\sum_{j=0}^{m-1}(\chi^{+}_{c,j}-\chi_{c,j})\right)$,
    	\item $\Sigma \cap B_n(x,\varepsilon,g)\subset \bigcup_{\gamma \in \Gamma'(\Sigma)} \gamma( [-1,1])$; and for every $\gamma \in \Gamma'(\Sigma)$
    	\item if $n>mq+c$, then $g^{k}\circ \sigma \circ \gamma$ is strongly $2\varepsilon$-bounded for $k\in \{0,c,c+q,\cdots,c+mq\}$,
    	\item if $n\leq mq+c$, then $g^{k}\circ \sigma \circ \gamma$ is strongly $2\varepsilon$-bounded for $k\in \{0,c,c+q,\cdots,c+(m-1)q\}$.
    \end{itemize}
     By the choice of $\varepsilon$ (Condition~\eqref{e.bound-varepsilon}), for any $\gamma\in \Gamma'(\Sigma)$, one has that
     \begin{equation}\label{e.up-to-n-initial}
     	\|D\big(g^j\circ\sigma\circ\gamma\big)\|\le 1,~~~\forall j=0,1,\cdots,n.
     \end{equation}
Now, we estimate $\# \Gamma'(\Sigma)$. For each $z\in \Sigma\subset K$ and $c_n(z)=c$, one has
\begin{align*}
	&\chi_{c}^{+}-\chi_c+\sum_{j=0}^{m-1}(\chi^{+}_{c,j}-\chi_{c,j})\\
	=~&\chi_{c}^{+}(z)-\chi_c(z)+\sum_{j=0}^{m-1}(\chi^{+}_{c,j}(z)-\chi_{c,j}(z))\\
    \leq~&m+1+\sum_{j=0}^{m-1} \log\|Dg^q(g^{qj+c}(z))\|-\log \|Dg^{mq}|E^u(z)\|+\log \|Dg^c(z)\|\\
	\leq~&m+1+n\left(\frac{1}{q}\int \log\|Dg^q\|{\rm d}\mu-\lambda^u(\mu,g)\right)+\frac{2n\delta(r-1)}{5}+q\log \|Dg\|, ~~\text{by \eqref{eq:Ly1}, \eqref{eq:Ly2}},\\
	\leq~&m+1+\frac{3n\delta(r-1)}{5}+q\log \|Dg\|,~~\text{by \eqref{e.q-Lyapunov}}.
\end{align*}
Note that the number of sets of the form  $\Sigma$ is at most
$$q\cdot \left(q(\log\|Df\|+\log \|Df^{-1}\|+4)\right)^{2m+2}.$$
Let $\Gamma_n$ denote the union of $\Gamma'(\Sigma)$ for all possible $\Sigma$. 
Then, for $n$ large enough we have
$$\# \Gamma_n\leq C_r^{m+1}\exp(\frac{3n\delta}{5})\cdot \exp (\frac{m+1+q\log \|Dg\|}{r-1})\cdot q\cdot \left(q(\log\|Df\|+\log \|Df^{-1}\|+4)\right)^{2m+2}.$$
Therefore, we have
\begin{align*}
	\limsup_{n\rightarrow \infty}\frac{1}{n} \log \# \Gamma_n&\leq \frac{\log C_r}{q}+\frac{3\delta}{5}+\frac{1}{q(r-1)}+\frac{2\log \left(q(\log\|Df\|+\log \|Df^{-1}\|+4)\right)}{q} \\
	&\leq \frac{\delta}{10}+\frac{3\delta}{5}+\frac{\delta}{10}+\frac{\delta}{5}\leq \delta,~~\text{by \eqref{e.q-norm}, \eqref{e.q-Cr}}.
\end{align*}
This completes the proof of the proposition.
\end{proof}

\begin{Remark}
	For the fixed point $x_0\in K$, consider a bounded curve $\sigma^s:[-1,1]\rightarrow M$ such that $\sigma^s_*=W^s_{\loc}(x_0)$, and the  $C^r $ diffeomorphism  $g^{-1}$.
	In this setting, we can also establish that a family of reparametrizations $\Gamma_n^s$ satisfies the conclusions of Proposition \ref{Pro:bound-reparametrizations}.
\end{Remark}		
\begin{Proposition}\label{Pro:metric-entropy-K}
For every $x_0\in K$ and every $Z\subset K$ with $\mu_{\xi(x_0)}(Z \cap W^u_{\rm loc}(x_0))>0$, we have 
$$h_{\mu}(g)\le \liminf_{n\to\infty}\frac{1}{n}\log r_{g}(\varepsilon,n, Z \cap W^u_{\rm loc}(x_0))+2\delta.$$
\end{Proposition}

\begin{proof}
By the choice of $K$, one has that
$$h_{\mu}(g)\le\lim_{\beta \to 0}\liminf_{n\to\infty}\frac{1}{n}\log r_{g}(\beta,n,Z\cap W^u_{\rm loc}(x_0))+\delta.$$
Now, consider an $(n,\varepsilon)$-separated set $\{y_1,\cdots,y_\ell\}$ of $Z \cap W^u_{\rm loc}(x_0)$ with minimal cardinality. 

For each $x=y_j$ with $1\le j\le\ell$, by Proposition~\ref{Pro:bound-reparametrizations}, there exists a family of reparametrizations $\Gamma_n=\Gamma_n(B_n(x,\varepsilon,g)\cap W^u_{\rm loc}(x_0)\cap K)$ such that
\begin{enumerate}
\item for every $\gamma\in \Gamma_n$, one has that  $\|g^i\circ\sigma\circ\gamma\|\le1,~ \forall 0\le i\le n-1$;
\item $\sigma^{-1}(K\cap W^u_{\rm loc}(x_0)\cap B_n(x,\varepsilon,g))\subset\bigcup_{\gamma\in\Gamma_n} \gamma([-1,1])$;
\item $\limsup\limits_{n\to\infty}\dfrac{1}{n}\log \#\Gamma_n\leq \delta$.
\end{enumerate}
Given $\beta>0$, choose a family of affine reparametrizations $\Theta$ with $\#\Theta\le \lceil \frac{2}{\beta}\rceil+1$ such that for any $\theta\in \Theta$, one has $\theta'<\beta/2$. We consider a new family of affine reparametrizations
$$\Gamma_n^*:=\Gamma_n\times \Theta:=\{\gamma\circ\theta:~\gamma\in\Gamma_n,~\theta\in\Theta\}.$$
Then, we have
$$\#\Gamma_n^*\le \#\Gamma_n\times \# \Theta\le ([2/\beta]+1 )\#\Gamma_n.$$
Since $\|D(g^i\circ\sigma\circ\gamma\circ\theta) \|\leq \frac{\beta}{2}$ for every $0\leq i<n$, the length of $g^i\circ\sigma\circ\gamma\circ\theta$ is smaller than $\beta$. 
Hence, for each $1\le j\le\ell$, choose $Z_j:=\{z_{j,k}\}$ with $1\le k\le \#\Gamma_n^*(B_n(x,\varepsilon,g))$ such that each $z_{j,k}$ is contained in one ${\rm Image}(g^i\circ\sigma\circ\gamma\circ\theta)$. Since the length of $g^i\circ\sigma\circ\gamma\circ\theta$ is smaller than $\beta$, the union $\widehat Z_n=\bigcup_{j=1}^\ell Z_j$ forms an $(n,\beta)$-spanning set. Thus, we have
$$r_{g}(\beta,n,Z\cap W^u_{\rm loc}(x_0))\leq \#\widehat Z_n \leq r_{g}(\varepsilon,n, Z \cap W^u_{\rm loc}(x_0))\cdot \#\Gamma_n^*\leq r_{g}(\varepsilon,n, Z \cap W^u_{\rm loc}(x_0))\cdot  \lceil \frac{2}{\beta}\rceil \cdot \#\Gamma_n.$$
Taking logarithms and limits, we conclude
$$h_{\mu}(g)\le\liminf_{n\to\infty}\frac{1}{n}\log r_{g}(\varepsilon,n, Z \cap W^u_{\rm loc}(x_0))+2\delta.$$
Thus the proposition is proved.
\end{proof}
\begin{Remark}
By considering the dynamics of $g^{-1}$, for every $x_0\in K$ and every $Z\subset K$ with $\mu_{\xi^s(x_0)}(Z \cap W^s_{\rm loc}(x_0))>0$, we have
$$h_{\mu}(g)=h_{\mu}(g^{-1})\leq \liminf_{n\rightarrow \infty}\frac{1}{n}\log r_{g^{-1}}(\varepsilon,n,Z\cap W^s_{\rm loc}(x_0))+2\delta.$$
\end{Remark}

\section{The main reduction theorem - geometry}\label{Sec:MRT-geo}

Section~\ref{Sec:MRT} has reduced the proof of Theorem~\ref{Thm:main-reduction} to some counting of reparametrizations. 
In this section we conclude the proof of Theorem~\ref{Thm:main-reduction} by using some geometric arguments completed by applications of Sard's lemma from \cite{BCS22}. 
We argue carefully to be able to conclude using just $\htop(f)>\lambda_{\min}(f)/r$.

\subsection{$su$-Quadrilaterals}
We adopt the definition of $su$-quadrilaterals from Buzzi-Crovisier-Sarig \cite[Definition 2.18]{BCS22}. 

 \begin{figure}[htp]
	\centering
	\includegraphics[width=0.3\linewidth]{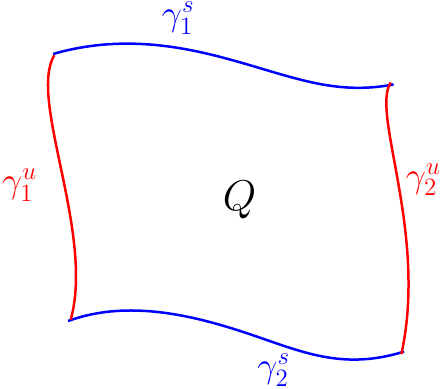}
	\caption{A $su$-quadrilateral $Q$ and its four boundary segments: $\gamma^s_1,\gamma^s_2$ (stable) and $\gamma^u_1,\gamma^u_2$ (unstable). }
	\label{Fig:Q}
\end{figure}

For a surface diffeomorphism $f$, given a hyperbolic periodic orbit $O$ of $f$, an $su$-quadri\-la\-te\-ral associated to $O$ is an open disk $Q$ in $M$ whose boundary is composed of four compact curves (or segments):
$$\partial Q=\gamma^s_1\cup \gamma^s_2\cup \gamma^u_1\cup \gamma^u_2,$$
where the stable boundary segments $\gamma^s_1$ and $\gamma^s_2$ are segments of the stable manifold of $O$ and  the unstable boundary segments $\gamma^u_1$ and $\gamma^u_2$ are segments of the unstable manifold of $O$. 
See Figure~\ref{Fig:Q}.

We will use without further comment the following simple topological fact. 
Recall the definition of topological crossing in Section~\ref{Sec:Katok-shadowing}.

\begin{Lemma}
Let $Q$ be some $su$-quadrilateral. If a segment $\gamma$ of some unstable manifold intersects both $Q$ and $M\setminus\overline{Q}$, then $\gamma$ topologically crosses at least one of the two stable boundary segments of $Q$.

The symmetric statement obtained by exchanging ``unstable'' and ``stable'' is also true.
\end{Lemma}

For $g$ close to $f$, denote by $O_g$ the hyperbolic continuation of $O$ with respect to $g$. 
Since the stable manifolds and the unstable manifolds of hyperbolic periodic orbits vary continuously with respect to diffeomorphisms, the $su$-quadrilaterals also have continuations. 
We denote by $Q(g)$ the continuation of $Q$ with respect to $g$.
 
We are in the setting of Theorem~\ref{Thm:main-reduction}. 
Recall that $\nu$ is an ergodic MME belonging to the ergodic decomposition of some MME $\mu_0$.
Without loss of generality, \textit{we assume that $$\lambda_{\min}(f)=\lambda^s(f)$$ in the next subsections}, otherwise we only need to consider $f^{-1}$. 
We use the following lemma, which is closely related to Buzzi-Crovisier-Sarig \cite[Proposition 2.19]{BCS22}. 
We add an additional {property}: the periodic orbit {$O$} is contained in a hyperbolic set whose entropy is large. 
Recall the constant $C_0$ from \eqref{eq:C0}.

\begin{Lemma}\label{Lem:quadrilateral}
In the setting of Theorem~\ref{Thm:main-reduction}, there are $su$-quadrilaterals $Q_1,Q_2,\cdots,Q_N$ associated to hyperbolic periodic orbits $P_1,\cdots,P_N$, which are homoclinically related with $\nu$, such that
$$\nu(\bigcup_{j=1}^N Q_j)>1/2,~~~{\rm Diam}(Q_i)<\frac{\varepsilon}{\|Df\|_{\rm sup}^{C_0}+2}<\varepsilon,~\forall 1\le i\le N.$$
Moreover, $\nu$ is homoclinically related with a horseshoe $\Lambda$ containing each $P_i$ for every $1\leq i\leq N$, and
$$\frac{h_{\rm top}(f|_\Lambda)}{\lambda^s(\Lambda,f)}>\frac{1}{r}.$$
\end{Lemma}
\begin{proof}
\cite[Proposition 2.19]{BCS22} gives periodic orbits $P_1',\dots,P_N'$ homoclinically related to the ergodic measure $\nu$ and associated $su$-quadrilaterals $Q_1',\dots,Q_N'$ satisfying the above claims, except for the inclusions of the periodic orbits in the horseshoe. 
Recall that
 $$\frac{h_{\nu}(f)}{\lambda^s(\nu,f)} {\geq} \frac{h_{\rm top}(f)}{\lambda^s(f)}>\frac{1}{r}.$$
By Katok's shadowing lemma (see Section~\ref{Sec:Katok-shadowing}), there exists a horseshoe $\Lambda'$ such that $\lambda^s(\Lambda',f)$ is close to $\lambda^s(\nu,f)$ and $h_{\rm top}(f|_{\Lambda'})$ is close to $h_{\nu}(f){=\htop(f)}$.  
Since the periodic orbits $P_1',\dots,P_N'$ are homoclinically related to $\Lambda$, one can find a larger horseshoe $\Lambda'\supset\Lambda\cup P'_1\cup\dots\cup P'_N$. 
Fix a hyperbolic periodic orbit $P:=P_1=\dots=P_N$  in $\Lambda$.  
By minimality of the dynamical foliations of $\Lambda'$ (the last sentence of Theorem \ref{Thm:hyperbolic-easy}), we see that one can approximate each $su$-quadrilateral $Q_i'$ by a $su$-quadrilateral $Q_i$ associated to $P_i$.
 \end{proof}
 
Recall that $\mu_0=\alpha\nu+(1-\alpha) \nu'$, so we have $\mu_0(\bigcup_{j=1}^N Q_j)>\alpha/2$. 
Choose a neighborhood $U$ of $\mu_0$ such that 
\begin{equation}\label{eq:ChooseU}
	\forall \mu'\in U,~\mu'(\bigcup_{j=1}^N Q_j)>\alpha/2.
\end{equation}
We choose a  $C^r$ neighborhood $\mathcal U_{\rm quad}$ such that for any $g\in\mathcal U_{\rm quad}$, for every $i\in \{1,\cdots,N\}$ the continuation of $P_i$ and $Q_i$ are well defined, denoted by $P_{i,g}$ and $Q_i(g)$.
Moreover, for each $i$, $P_{i,g}$ is homoclinically related with $O_g$ and for every $\mu \in U$ one has
 $$\mu(\bigcup_{j=1}^N Q_i(g))>\alpha/2,~~~{\rm Diam}(Q_i(g))<\frac{\varepsilon}{\|Dg\|_{\rm sup}^{C_0}+1}<\varepsilon,~\forall 1\le i\le N.$$
 Now, we let 
 \begin{equation}
 	\cU=\cU'\cap \cU_{\rm quad}=\cU_{\rm quad} \bigcap \cU_{\rm ent}\bigcap \cU_{\rm con} \bigcap \cU_{\rm boun}
 \end{equation}
 
The following proposition establishes the topological crossing of typical unstable or stable manifolds with the boundary of the $su$-quadrilaterals, which constitutes one of the main ingredients of Theorem \ref{Thm:main-reduction}.
 
 \begin{Proposition}\label{Pro:topological-cross}
 In the setting of Theorem~\ref{Thm:main-reduction}, there exists a set $K_0$ of positive $\mu$-measure such that for every point $x\in K_0$, there are $n\ge 0$ and an $su$-quadrilateral $Q_{i_x}(g)$ satisfying the following 
 \begin{itemize}
 \item $g^n(x)$ is contained in the interior of $Q_{i_x}(g)$;
 \item $W^u(g^n(x))$ intersects both  $Q_{i_x}(g)$ and $M\setminus \overline{Q_{i_x}(g)}$.
 \end{itemize}
 \end{Proposition}
 
\begin{proof}
 We refer to Section~\ref{Sec:Ideas} for a comparison with the approach of Buzzi-Crovisier-Sarig \cite{BCS22}.
 We take $K'=K\cap\big(\bigcup_{j=1}^N Q_j(g)\big)$. 
 Since $\mu(K)>1-\alpha/2$ and $\mu(\bigcup_{j=1}^N Q_i(g))>\alpha/2$, one has that $\mu(K')>0$. 
 By the Birkhoff ergodic theorem, there exist $K''\subset K'$ and $n_0\in \NN$ such that $\mu(K'')>0$ and
 $$\forall n\geq n_0,~\forall x\in K'',~\# \left\{0\leq k<n: g^k(x)\in  \bigcup_{j=1}^N Q_i(g)\right\}>\frac{n\alpha}{2}.$$
 We now take $K_0\subset K''$ such that for every $x\in K_0$, we have  $\mu_{\xi(x)}(K''\cap W^u_{\loc}(x))>0 $.
 
 We will prove the proposition by contradiction. 
 We assume that there exists a point $x_0\in K_0$ such that for every $n\in \NN$, $g^n(W^u(x_0))$ is either totally contained in one of the $su$-quadrilaterals or disjoint from all these $su$-quadrilaterals. 
 
 By the choice of $K''$ and Proposition \ref{Pro:metric-entropy-K}, we have
 \begin{equation}\label{eq:Prop410}
 	h_{\mu}(g)\leq \liminf_{n\rightarrow +\infty} \frac{1}{n} r_{g}(\varepsilon,n,K''\cap W^u(x_0))+\delta.
 \end{equation}
 Now, we estimate $r_{g}(\varepsilon,n,K''\cap W^u_{\rm loc}(x_0))$.
 For $i=0,1,\cdots,n-1$, if $g^i(x_0)\in  \bigcup_{j=1}^N Q_i(g)$, then $g^i(W^u(x_0)\cap K_0)$ is totally contained in one $su$-quadrilateral.
 In this case, we have
 $${\rm Diam}(g^i(W^u(x_0)\cap K''))<\frac{\varepsilon}{\|Dg\|_{\rm sup}^{C_0}+1}.$$
 Therefore, $g^{i+j}(K''\cap W^u_{\rm loc}(x_0))$ is contained in a ball of size $\varepsilon$ for all $0\le j\le C_0$. Let
 $$J_1:=\{i:g^i(x_0)\in~\textrm{some $su$-quadrilateral}\},~~~\widehat J:=J_1+[0,C_0).$$
 Write $\mathbb N\cup\{0\}\setminus\widehat J$ as a disjoint union of maximal subintervals 
 $$[a_1,b_1]<[a_2,b_2]<\cdots~~~\textrm{with}~a_{s+1}>b_s+C_0,~\forall s\ge 1.$$
 Denote by $L_s$ an $(\varepsilon/2,b_s-a_s)$-spanning set of $M$ with minimal cardinality. 
 Recall the choice of $\varepsilon$, $N_0$ and $C_0$ (see \eqref{e.q-norm}, \eqref{e.bound-U} and \eqref{eq:C0}). 
 For all $k=sN_0+\tau>0$ with $\tau\in [0,N_0)$, we have
 $$r_g(\varepsilon,k,K''\cap W^u_{\rm loc}(x_0))\leq r_g(\varepsilon,k)\leq e^{(s+1)N_0 (h_{\rm rop}(f)+\delta)}\leq  e^{(\frac{k}{N_0}+1)N_0(h_{\rm rop}(f)+\delta)}\leq e^{k(h_{\rm rop}(f)+\delta)+\delta C_0}.$$
 As discussed above,
 $$\# L_s\le {\rm e}^{\delta C_0}\cdot\exp\big((h_{\rm top}(f)+\delta)(b_s-a_s)\big).$$
 Consider $n=b_t$ for some large $t$. By the Birkhoff ergodic theorem, we have
 $$\sum_{i=1}^t(b_i-a_i+1)<(1-\frac{\alpha}{2})n,~~~\textrm{and}~~~t<\frac{n}{C_0}+1.$$
 Thus, we obtain
 \begin{align*}
 r_{g}(\varepsilon,n,K''\cap W^u_{\rm loc}(x_0))&\le\prod_{i=1}^t\# L_i\le \prod_{i=1}^t{\rm e}^{\delta C_0}\cdot\exp\big((h_{\rm top}(f)+\delta)(b_s-a_s)\big)\\
 &\le {\rm e}^{n\delta+\delta C_0}\exp\big((1-\frac{\alpha}{2})(h_{\rm top}(f)+\delta)n\big).
 \end{align*}
 This implies that
 $$\liminf_{n\to\infty}\frac{1}{n}\log r_{g}(\varepsilon,n,K''\cap W^u_{\rm loc}(x_0))\le 2\delta+(1-\frac{\alpha}{2})h_{\rm top}(f).$$
 By combining inequality \eqref{eq:Prop410}, we deduce that $h_\mu(g)\leq h$ (see \eqref{eq:hU'} for $h$).
 Thus, we obtain a contradiction.
 \end{proof}
\begin{Remark}\label{Re:stable}
	By considering the dynamics of $g^{-1}$ and reducing $K_0$ if necessary, we can assume that $\mu(K_0)>0$ and that for every $x\in K_0$, there exist $n:=n_x\geq 0$ and an $su$-quadrilateral $Q_{i_x}(g)$ such that  $g^{-n}(x)$ is contained in the interior of $Q_{i_x}(g)$ and  $W^s(g^{-n}(x))$ intersects both  $Q_{i_x}(g)$ and $M\setminus \overline{ Q_{i_x}(g)}$.
\end{Remark}
 
 \subsection{Unstable lamination of $\mu$ and stable manifold of $O_g$}
In this subsection, we use that the ergodic MME $\mu$ has an unstable lamination with a large transverse dimension to find a transverse intersection
 \begin{equation}\label{eq-mu-prec-O}
    W^u(x_0)\pitchfork W^s(P_i) \ne\emptyset
  \end{equation}
for any $x_0\in K$ and some periodic orbit $P_i$ corresponding to some $su$-quadrilateral $Q_i$ from Lemma~\ref{Lem:quadrilateral}. Since $\mu(K_0)>0$, this will give: $\mu\preceq O_g$.

Recall that, by  Proposition~\ref{Pro:topological-cross}, for each $x_0\in K_0$ there exist integers $n>0$ and $i_0$ such that $W^u(g^n(x_0))$ {topologically crosses the boundary of one $su$-quadrilateral} $Q_{i_0}$.
Since the entropy of $\mu$ is larger than $h$, by Lemma~\ref{Lem:h-neighborhood}, one has that
$$h_{\mu}(g)\ge h>\frac{\lambda^s(g)}{r}\ge\frac{\lambda^s(\mu,g)}{r}.$$
By the choice of $K$ in Section~\ref{Sec:choose-several}, there is a horseshoe $\Lambda_0$ and a periodic point $p\in\Lambda_0$ such that
\begin{itemize}
 \item $h_{\rm top}(g|_{\Lambda_0})/\lambda^s(g,\Lambda_0)>1/r$, hence the transverse dimension of the unstable lamination of  $\Lambda_0$ is large.
 \item the orbit ${\rm Orb}(p)$ is homoclinically related with $\mu$ and the unstable manifold of ${\rm Orb}(p)$ accumulates to the unstable manifold of $x_0$: ${\rm Closure}(W^u(x_0))\subset {\rm Closure}(W^u({\rm Orb}(p)))$. 
 \end{itemize}
 Since $W^u(g^{n}(x_0))$ topologically crosses the boundary of $Q_{i_0}$, and the unstable manifolds of ${\rm Orb}(p)$ $C^1$-accumulate the unstable manifold of $x_0$, one can select $q\in {\rm Orb}(p)$ such that $W^u(q)$ topologically crosses the boundary of $Q_{i_0}$.
 Since unstable manifolds are equal or disjoint,  $W^u(q)$ topologically crosses at least one of the stable boundary of $Q_{i_0}(g)$, which we write as the image of a $C^r$ curve $\gamma:[-1,1]\rightarrow M$.

 \begin{figure}[htp]
	\centering
	\includegraphics[width=1.0\linewidth]{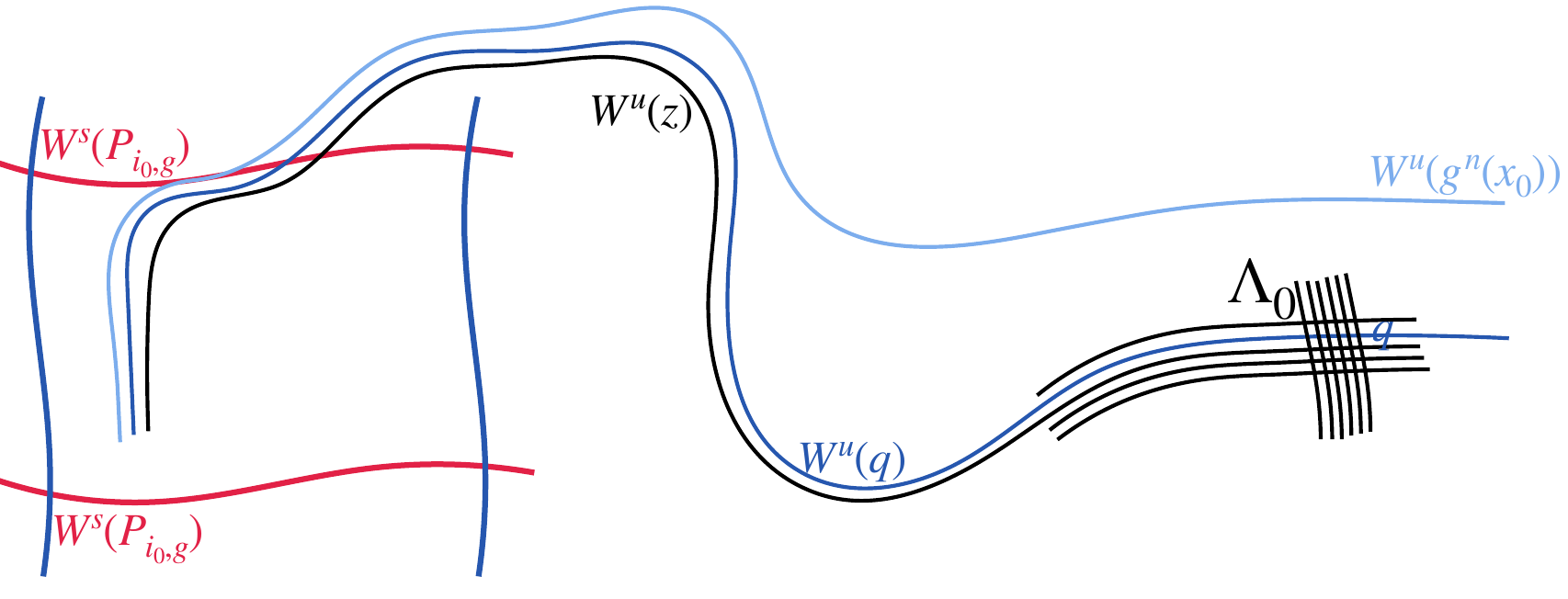}
	\caption{The transverse intersection of $W^u({\rm Orb}(p))$ and  stable boundaries of $Q_{i_0}(g)$ }
	\label{Fig:U}
\end{figure}
 
 By minimality of the unstable lamination of $\Lambda_0$, this is true not only for $p$ but for any $z\in \Lambda_{0}$: the unstable manifold of some iterate of $z$ intersects with $\gamma$. We refer to Figure \ref{Fig:U}. 

 Since the transverse dimension of the unstable lamination of $\Lambda_0$ is large, we can apply Theorem \ref{Thm:basic-set-sard} and deduce that ${\rm Image}(\gamma) \pitchfork W^u(z)$ for any $z\in \Lambda_{0}$. 
 Hence, we conclude that $W^u({\rm Orb}(p))$ intersect $W^s(P_{i_0,g})$ transversely, i.e., \eqref{eq-mu-prec-O} is proved.
 
 \subsection{Stable lamination of the $\mu$ and unstable manifold of $O_g$} \label{SEC:Stable}
 We know that, for any $x_0\in K_0$, for some integers $n>0$ and $i_0$, $W^s(f^{n} x_0)$ topologically crosses the boundary of $Q_{i_0}$. We need to get transversality:
 \begin{equation}\label{eq-O-prec-mu}
    W^s(x_0)\pitchfork W^u(P_i) \ne\emptyset.
  \end{equation}
In general,  the stable lamination of $\mu$ can have a small transverse dimension. 
For this reason, we use the additional information from Lemma~\ref{Lem:quadrilateral} (in comparison to \cite[Proposition 2.19]{BCS22}): the unstable boundary of the $su$-quadrilaterals $Q_1,\dots,Q_N$ is accumulated by the unstable lamination of the horseshoe $\Lambda$. 
Thus $W^s(f^{-n}x_0)$ will have to topologically cross not just the unstable boundary of $Q_i$, but this unstable lamination. 
We will conclude by noting that the unstable lamination of $\Lambda$ has large transverse dimension.
 
By Proposition~\ref{Pro:topological-cross} and Remark~\ref{Re:stable}, we know that there is a measurable set $K_0$ with $\mu(K_0)>0$ such that for every point $x\in K_0$, there is an integer $n\ge 0$ and an $su$-quadrilateral $Q_{i_0}(g)$ associated to $P_{i_0,g}$ such that 
 \begin{itemize}
 \item $g^{-n}(x)$ is contained in the {open disk} $Q_{i_0}(g)$;
\item  $W^s(g^{-n}(x))$ intersect both  $Q_{i_0}(g)$ and $M\setminus {\overline{Q_{i_0}(g)}}$.
 \end{itemize}

Since $O$ is homoclinically related with the measure $\nu$ of maximal entropy of $f$, by Lemma \ref{Lem:quadrilateral}, there is a horseshoe $\Lambda$ of $f$ such that $O\subset\Lambda$ and 
$$\frac{h_{\rm top}(f|_\Lambda)}{\lambda^s(\Lambda,f)}>\frac{1}{r}.$$

\begin{Claim}
	There is a $C^r$ neighborhood $\mathcal U_{\rm dim}$ of $f$ such that for any $g\in \mathcal U_{\rm dim}$, one has that
	\begin{equation}\label{eq:Udimen}
		\frac{h_{\rm top}(g|_{\Lambda_g})}{\lambda^s(\Lambda_g,g)}>1/r,
	\end{equation}
	where $\Lambda_g$ is the hyperbolic continuation of $\Lambda$ with respect to $g$.
\end{Claim}

This is obvious since $\htop(g|\Lambda_g)=\htop(f|\Lambda)$ and $g\mapsto \lambda^s(\Lambda_g,g)$ is continuous.

\medbreak

Take a $C^r$ curve $\gamma:~[-1,1]\to M$  such that ${\rm Image}(\gamma)$ is a compact part of $W^s(g^{-n}(x))$  that topologically crosses the unstable boundary of $Q_{i_0}(g)$.
Recall that the unstable boundary of  $Q_{i_0}(g)$ is contained in $W^u(P_{i,g})$ with $P_{i,g}\subset\Lambda_g$.
Condition \eqref{eq:Udimen} ensures that the horseshoe $\Lambda_g$ has an unstable lamination with transverse dimension strictly larger than $1/r$. 

To prove \eqref{eq-O-prec-mu}, it is enough apply Theorem~\ref{Thm:basic-set-sard} for $g$ close to $f$. 
We only need to verify the following claim.
\begin{Claim}
For any point $z\in\Lambda_g$, $\gamma$ topologically crosses $W^u({\rm Orb}(z),g)$.
\end{Claim}
\begin{proof}[Proof of the claim]
Consider a point $y\in\Lambda_g$ such that $W^u(y,g)$ accumulates on $W^u(P_{i_0,g},g)$. 
Since $\gamma$ topologically crosses $W^u(P_{i_0,g},g)$, it must topologically cross $W^u(y)$.

Since $\gamma$ has no intersection with the stable manifold of $P_{i_0,g}$, we choose $\gamma^u_{i_0}$ to a connected component of the unstable boundary of $Q_{i_0}(g)$ such that $\gamma$ intersects both the left and right sides of $\gamma^u_{i_0}$. 
Therefore, one can take
\begin{itemize}
\item $\gamma_1^s$ and $\gamma_2^s$ {the two stable boundary segments} of $P_{i_0,g}$
\item $\gamma^u_y$ from the unstable manifold of $y$ which is very close to $\gamma^u_{i_0}$.
\end{itemize}

\begin{figure}[htp]
	\centering
	\includegraphics[width=1.0\linewidth]{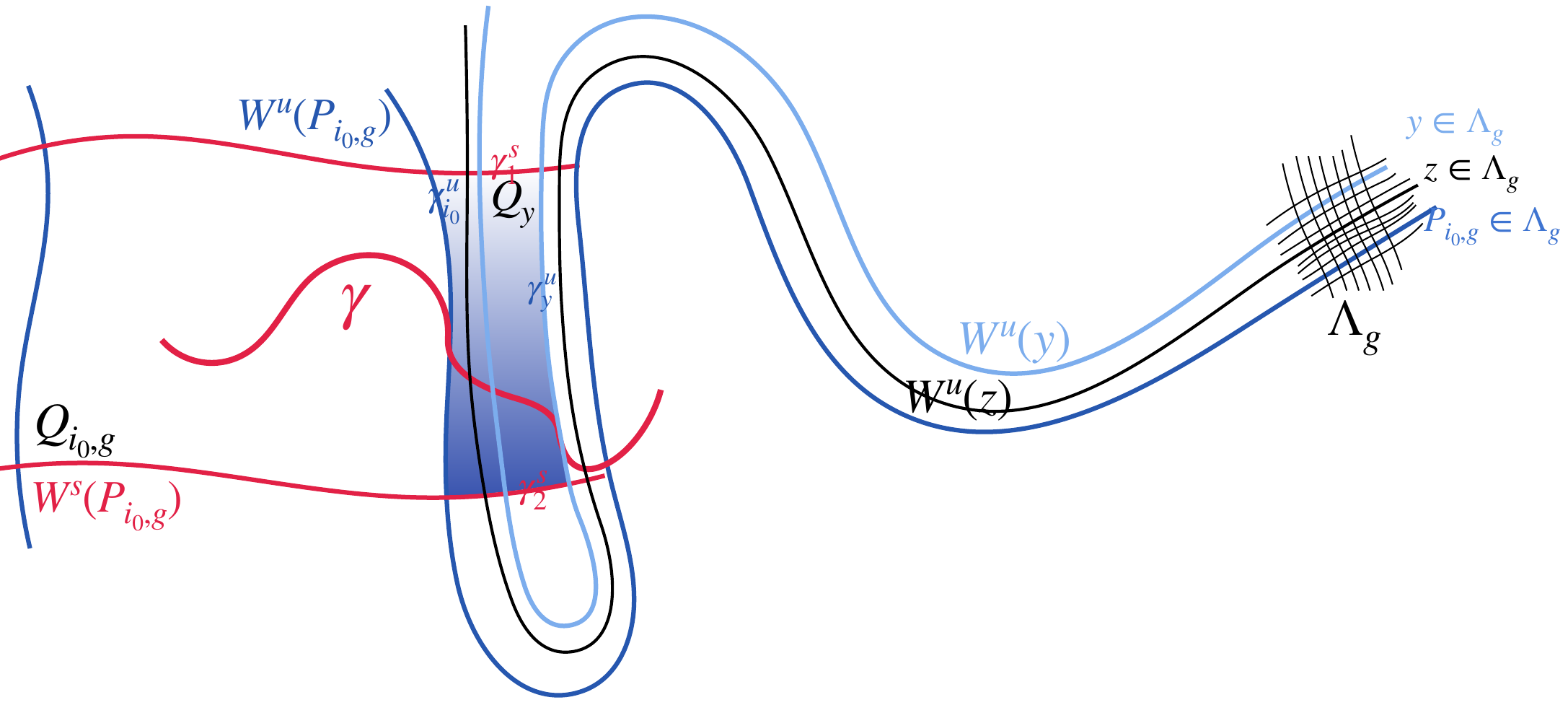}
	\caption{Small $su$-quadrilateral $Q_y$ bounded by $\gamma^u_{i_0}$, $\gamma^u_y$ and $\gamma_1^s$, $\gamma_2^s$.}
	\label{Fig:S}
\end{figure}

This allows us to define a small $su$-quadrilateral $Q_y$ bounded by $\gamma^u_{i_0}$, $\gamma^u_y$ and $\gamma_1^s$, $\gamma_2^s$. 
It follows that $\gamma$ intersects the two connected components of the unstable boundary $\gamma^u_{i_0}\cup \gamma^u_y$ of $Q_y$. Moreover, one can assume that the unstable manifold of $y$ contains many connected components that are contained in $Q_y$ and accumulate on $\gamma^u_{i_0}$. 
We refer to Figure \ref{Fig:S}. 

For any point $z\in\Lambda_g$, the unstable manifolds of ${\rm Orb}(z)$ accumulate the unstable manifolds of ${\rm Orb}(y)$. 
Thus, the unstable manifolds of ${\rm Orb}(z)$ intersect $Q_y$ and one part of the unstable manifolds of ${\rm Orb}(z)$ cuts $Q_y$ into two disjoint connected smaller $su$-quadrilaterals. 
Thus, $\gamma$ intersects the unstable manifolds of ${\rm Orb}(z)$. 
This completes the proof of the claim.
\end{proof}

By Theorem~\ref{Thm:basic-set-sard}, one has that $\gamma\pitchfork W^u({\rm Orb}(y))\neq\emptyset$ for any point $y\in\Lambda$. Since $P_{i_0,g} \subset\Lambda_g$, one has that $\gamma$ intersects $W^u(P_{i_0,g})$ transversely.
Thus, we conclude that $W^s(x_0)$ intersects $W^u(P_{i_0,g})$ transversely.

From the discussion in Section \ref{Sec:MRT}, we have completed the proof of Theorem \ref{Thm:main-reduction}.

\section{Proof of the Main Theorems and Constructions}\label{Sec:ProofMainThms}

We deduce our main results, Theorems \ref{Thm:uniform-finite} to \ref{Thm:HomoRel}, from Theorem~\ref{Thm:main-reduction}.

\subsection{Proof of Theorem~\ref{Thm:HomoRel} and Corollary~\ref{Cor:LimitDisjoint}}
We establish the relevant consequences of Theorem~\ref{Thm:main-reduction}.

\begin{proof}[Proof of Theorem~\ref{Thm:HomoRel}]
Let $f_n\to f$ and $\mu_n\to\mu_0$ be as in the assumptions  of the theorem.
Since $\htop(f)=\lim_n h_{\mu_n}(f_n)$, Theorem~\ref{Thm:continuity-Lyapunov-exponent} implies that  $\mu_0=\lim_n\mu_n$ is an MME of $f$.
In particular, the ergodic decomposition of $\mu_0$ is carried by the ergodic MMEs. Now, there are only finitely many such measures and they are all hyperbolic. Let $\nu$ be one of them and let $O$ be a hyperbolic periodic orbit with $\nu\homrel{f}O$. Letting $O_n$ be the hyperbolic continuation of $O$, Theorem~\ref{Thm:main-reduction} shows that $\mu_n\homrel{f_n}O_n$ for all large $n$.
\end{proof}

Our proofs of Theorems A, B, and C will rely on the following consequences of Theorem~\ref{Thm:HomoRel}. In fact, we will no longer need to refer Theorems \ref{Thm:HomoRel} or \ref{Thm:main-reduction} once we have Corollaries \ref{Cor:LimitDisjoint}~and~\ref{Cor:LimitDisjointMME}.

\begin{Corollary}\label{Cor:LimitDisjoint}
Let $f$ be a $C^r$ surface diffeomorphism satisfying \eqref{eqLargeEntropy} {\rm(}for some $r>1${\rm)}. 
Let $f_n\to f$ in the $C^r$ topology. 
For each $n \ge 1$, let $m^1_n,\dots,m^k_n$ be ergodic, hyperbolic measures satisfying, for each $1\le i\le k$,
 \begin{itemize}
  \item[(*)]  $m^i_n$ is not homoclinically related to any $m^j_n$ with $j\in\{1,\dots,k\}\setminus\{i\}$;
  \item[(**)] $\lim_n h_{m^i_n}(f_n) = \htop(f)$;
  \item[(***)] the weak$\ast$-limit $\mu^i:=\lim_n m^i_n$ exists.
\end{itemize}
Then each limit $\mu^i$ is an MME and any pair of measures $\mu^i$ and $\mu^j$ {\rm(}$i\ne j${\rm)} are mutually singular. 
In particular, denoting by $\mathcal E^i$ the finite support of the ergodic decomposition of $\mu^i$, we have:
 \begin{equation}\label{eq:EqClass}
    k\le \sum_{i=1}^k \#\mathcal E^i\le N_f.
  \end{equation}
\end{Corollary}

\begin{proof}
Let $f_n\to f$, $k\ge1$, and $m^i_n$ be as in the assumptions above.
Together with (**), Theorem~\ref{Thm:continuity-Lyapunov-exponent} implies that each $\mu^i$ is an MME. 
Let us prove that $\mu^i=\lim_n m^i_n$ and $\mu^j=\lim_n m^j_n$ are mutually singular for any pair $1\le i\ne j\le k$. 

Let $\nu,\nu'$ belong to the supports of the ergodic decompositions of $\mu^i$ and $\mu^j$, respectively. 
Let $O,O'$ be hyperbolic periodic orbits of $f$ with the following homoclinc relations: $O\homrel{f}\nu$ and $O'\homrel{f}\nu'$. 
Together with (**), Theorem~\ref{Thm:HomoRel} implies that the hyperbolic continuations $O_n,O_n'$ satisfy: $O_n\homrel{f_n} m^i_n$ and $O'_n\homrel{f_n} m^j_n$. 
Now, if $\nu$ and $\nu'$ were equal,  $O$ and $O'$ would be homoclinically related. 
Hence, by robustness of the homoclinic relation among hyperbolic periodic orbits, we would get that, for all large $n$, $O_n \homrel{f_n} O'_n$  and therefore $m^i_n \homrel{f_n}  m^j_n$, contradicting assumption~(*). 
The contradiction proves that $\nu\ne\nu'$, hence $\mu^i,\mu^j$ are mutually singular, as claimed.

Since the supports of the ergodic decompositions of the measures $\mu^1,\dots,\mu^k$ are pairwise disjoint, nonempty, and included in the set of ergodic MMEs of $f$, equation \eqref{eq:EqClass} follows.
\end{proof}

We will mostly use the following restriction to the case of ergodic MMEs.

\begin{Corollary}\label{Cor:LimitDisjointMME}
Let $f$ be a $C^r$ surface diffeomorphism satisfying \eqref{eqLargeEntropy} {\rm(}for some $r>1${\rm)}. 
Let $f_n\to f$ in the $C^r$ topology. 
Fix some $k\ge1$. Assume that, for each $n\ge1$, there are $k$ distinct ergodic MMEs: $\mu^1_n,\dots,\mu^k_n$. Assume moreover, that, for each $1\le i\le k$, the weak$\ast$-limit $\mu^i:=\lim_n \mu^i_n$ exists.

Then each limit $\mu^i$ is an MME and any pair of measures $\mu^i$ and $\mu^j$ {\rm(}$i\ne j${\rm)} are mutually singular. 
In particular, denoting by $\mathcal E^i$ the finite support of the ergodic decomposition of $\mu^i$, we have:
 \begin{equation}\label{eq:EqClassMME}
    k\le \sum_{i=1}^k \#\mathcal E^i\le N_f.
  \end{equation}
\end{Corollary}

\begin{proof}
Let $f_n\to f$, $k\ge1$, and $\mu^i_n$ be as in the assumptions above.
We apply Corollary~\ref{Cor:LimitDisjoint} to $m^i_n:=\mu^i$ for $1\le i\le k$. Since these ergodic MMEs are distinct they are not homoclinically related by Theorem~\ref{Thm:BBCS-revisited}, so condition (*) holds. Since the topological entropy is continuous in the large entropy surface setting, it follows that $\lim_n h_{\mu^i_n}(f) = \lim_n \htop(f_n)=\htop(f)$, i.e., (**). The convergence (***) is part of the assumptions. Thus we can apply Corollary~\ref{Cor:LimitDisjoint}.
\end{proof}

\subsection{Proof of Theorem~\ref{Thm:uniform-finite}}
Let $f$ be a $C^r$ diffeomorphism in the large entropy, surface setting. 
It is enough to show the inequality $1\le N_{f_n}\le N_f$ for an arbitrary sequence $f_n\to f$ in the $C^r$ topology. We fix such a sequence.

First, Theorem~\ref{Thm:BBCS-revisited} shows that $N_{f_n}\ge1$. To see that $N_{f_n}\le N_f$ for all large $n$, we assume by contradiction that there is $k>N_f$ and a subsequence such tha $N_{f_{n_i}}\ge k$ for all $i$.
Then Corollary~\ref{Cor:LimitDisjointMME} and especially \eqref{eq:EqClassMME} show that $k\le N_f$, a contradiction.

Theorem~\ref{Thm:uniform-finite} is proved. \qed

\subsection{Proof of Theorem~\ref{Thm:equal}}
Let $f_n \to f$ be as in the assumptions of the theorem. 
We first assume item (1), i.e., $N_f=N_{f_n}$ for all $n\ge1$, and deduce item (2).

Corollary~\ref{Cor:LimitDisjointMME} applied to the ergodic MMEs $m^1_n,\dots,m^k_n$ ($k:=N_f$)  yields mutually singular, hence distinct MMEs $\mu^i=\lim_n m^i_n$, $1\le i\le k$.
Since $k=N_f$, \eqref{eq:EqClassMME} becomes $k=\sum_{i=1}^k\#\mathcal E^i$. 
It follows that $\#\mathcal E^1=\dots=\#\mathcal E^k=1$. 
This means that each $\mu^i$, $1\le i\le k$, is ergodic. Since $k=N_f$ we have all the ergodic MMEs of $f$. 
Item (2) is proved.

\medbreak

We now assume that item (1) fails and deduce that item (2) also fails. 
Thus, maybe after passing to some subsequence, for all $n$, $N_{f_n}$ is equal to some integer $k<N:=N_f$ for all  $n$. 
But $N$ distinct ergodic MMEs of $f$ cannot be obtained as limits of $k$ distinct ergodic MMEs of $f_n$: item (2) fails. This proves Theorem~\ref{Thm:equal}. \qed

\subsection{Proof of Theorem~\ref{Thm:ergodicLimit}}

We let $f_n \to f$, $\mu$, and $(\mu_n)$ as in the assumptions of the theorem. 
We set $m^1_n:=\mu_n$. 
Let $k:=N_f$. 
By assumption, $k=N_{f_n}$ for all $n\ge1$.
By Theorem~\ref{Thm:BBCS-revisited}, the $k$ ergodic MMEs of each $f_n$ are pairwise non homoclinically related. 
Since at most one can be homoclinically related to $\mu_n$, one can find $k-1$ ergodic MMEs $m^2_n,\dots,m^k_n$ of $f_n$ such that, for all $1\le i\ne j\le k$, $m^i_n\not\homrel{f_n}m^j_n$. 
Note that $\lim_n h_{m^i_n}(f_n)=\htop(f)$ by assumption for $i=1$, since the topological entropy is continuous for $i\ge2$. 

Let us check the assumptions (*)- (**)-(***) of Corollary~\ref{Cor:LimitDisjoint}.
We already know for (*) and (**). 
To ensure (***) it is enough to pass to a subsequence (this does not change the limit $\mu=m^1=\lim_n m_n^1$). 
Thus Corollary~\ref{Cor:LimitDisjoint} applies, showing that all weak$\ast$-limits $m^1,\dots,m^k$ are MME whose ergodic decompositions satisfy equation~\eqref{eq:EqClass}. 
Since $k=N_f$, this equation implies that each $\mathcal E^i$ has single element, i.e., each $m^i$ is ergodic. 
In particular $\mu=m^1$ is ergodic.
Theorem~\ref{Thm:ergodicLimit} is proved.\qed

\subsection{Construction of examples}\label{Sec:Construction}

We prove the following slight strenthening of Proposition~\ref{Prop:Construction}:

\begin{Proposition}\label{Prop:Example-detailed}
Given any two integers $1\le m\le n$ and any surface, there is a smooth family of $C^\infty$-diffeomorphisms $f_t$, $|t|<\epsilon$, of that surface such that $N_{f_0}=n$ and $N(f_t)=m$ for all $t>0$. 
Moreover, the family $(f_t)$ satisfies property (*) in Theorem~\ref{Thm:ergodicLimit} even when $m<n$. 
\end{Proposition}

\medbreak

We first show that there exists a family $(g_t)_{0\le t<\epsilon}$ such that:
 \begin{enumerate}
  \item $g_0$ contains a basic piece $\Lambda_0$ which carries the unique MME;
  \item $\Lambda_0$ presents a first homoclinic tangency;
  \item for all $0\le t<\epsilon$, $g_t$ has a unique MME;
  \item for all $t>0$, $\htop(g_t)>\htop(g_0)>0$.
\end{enumerate}

Indeed, the starting point of Palis-Yoccoz \cite[page 3]{Palis} is a $C^\infty$ diffeomorphism of a disk, which is the identity near the boundary and satisfies the above items 1 and 2.  
Item 3 follows from Theorem~\ref{Thm:uniform-finite} and the fact that $N_{g_0}=1$ (perhaps after reducing $\epsilon$).

It remains to prove the last item. for $t>0$, the hyperbolic continuation $\Lambda_t$ of $\Lambda_0$ has a transverse homoclinic intersection and therefore is contained in a strictly larger horseshoe. 
By the analoguous result for transitive subshifts of finite type, the larger horseshoe must have larger entropy. 
This proves item 4, maybe after reducing again $\epsilon>0$.

We have a family $(g_t)_{0\le t<\epsilon}$ with properties (1)-(4). 
We now build the example.
Give an arbitrary surface $S$ we define a family $f_t:S\to S$ in the following way.
We find disjoint closed disks $D_1,\dots,D_n\subset S$. We let $f_t$ be the identity on $S\setminus (D_1\cup\dots\cup D_n)$. 
We let $f_t=g_t$ on $D_i$ for $1\le i\le m$. We let $f_t=g_0$ on $D_i$ for $m<i\le n$. 
 
 \medbreak
 
Finally we observe that each ergodic MME of $f_t$ converges to an ergodic MME of $f_0$, i.e., Theorem~\ref{Thm:ergodicLimit} holds, irrespective of the equality $N_f=N_{f_t}$.
\qed

\vskip 5pt

\flushleft{\bf J\'{e}r\^{o}me Buzzi} \\
\small Laboratoire de Math\'ematiques d'Orsay,  CNRS - UMR 8628, Universit\'e Paris-Saclay,  Orsay 91405, France\\
\textit{E-mail:} \texttt{jerome.buzzi@universite-paris-saclay.fr}\\

\flushleft{\bf Chiyi Luo} \\
\small School of Mathematics and Statistics, 
Jiangxi Normal University, Nanchang,   330022, P. R. China\\
\textit{E-mail:} \texttt{luochiyi98@gmail.com}\\

\flushleft{\bf Dawei Yang} \\
\small School of Mathematical Sciences,  Soochow University, Suzhou, 215006, P.R. China\\
\textit{E-mail:} \texttt{yangdw@suda.edu.cn}\\
\end{document}